\documentclass[reqno]{amsart}
\usepackage{amssymb,amsmath,amsthm,amscd,latexsym,longtable,amsfonts}
\usepackage{mathtools}
\usepackage[english]{babel}
\usepackage[T1]{fontenc}
\usepackage{inputenc}
\usepackage[arrow,matrix]{xy}
\usepackage{amsaddr}
\usepackage{graphicx}
\usepackage{cite}
\newtheorem{thm}{Theorem}[section]

\newtheorem{prop}[thm]{Proposition}

\newtheorem{defn}[thm]{Definition}

\newtheorem{rem}[thm]{Remark}
\numberwithin{equation}{section}

\newcommand{\Der}{\textrm{Der}}

\begin{document}

\title[Solvable extensions of the quasi-filiform Leibniz]{Solvable extensions of the naturally graded quasi-filiform Leibniz algebras}

\author{Abdurasulov K.K.\textsuperscript{1,2} and Adashev J.Q.\textsuperscript{1,2}}

\address{\textsuperscript{1}Institute of Mathematics, Uzbekistan  Academy of  Sciences, 9 University street, 100174.  abdurasulov0505@mail.ru, adashevjq@mail.ru}
\address{\textsuperscript{2}Chirchiq State Pedagogical Institute of Tashkent region, 104 Amir Temur street, 111700, Tashkent,  Uzbekistan.}

\begin{abstract} The present article is a part of the study of solvable Leibniz algebras with a given nilradical. In this paper solvable Leibniz algebras, whose nilradicals is naturally graded quasi-filiform algebra and the complemented space to the nilradical has one dimension, are described up to isomorphism.
\end{abstract}

\maketitle






\section{Introduction}

Leibniz algebras are a non-antisymmetric analogue of Lie algebras \cite{Bloh, Loday}, which makes that every Lie algebra is a Leibniz algebra. Since then many analogs of important theorems in Lie theory were found to be true for Leibniz algebras, such as the analogue of Levi's theorem which was proved by Barnes \cite{Bar}.

It is well known from the structure theory of finite-dimensional Lie algebras over a field of characteristic zero that there are three main classes of Lie algebras: semisimple, solvable, and those that are neither semisimple nor solvable \cite{cartan}. By the Levi-Mal'cev theorem, algebras from the third class are represented as a semidirect sum of semisimple and solvable algebras. Recall that the classification of semisimple Lie algebras was completely obtained by Cartan: any finite-dimensional semisimple Lie algebra over a field of characteristic zero can be decomposed into a direct sum of prime ideals. Thus, the problem of classifying solvable algebras, which are the second of the three main classes of Lie algebras, has become very topical. Thanks to the method of G.M. Mubarakzyanov \cite{Mub} finite-dimensional solvable Lie algebras can be reconstructed using their maximal nilpotent ideals and outer derivations of these ideals. Mubarakzyanov's method also extends to solvable Leibniz algebras. The fact that the operator of right multiplication by an element of the complemented subspace to the nilradical of a solvable Leibniz algebra is the outer derivation of the nilradical is a key role in this description. In recent years, significant results have been obtained in the classification of solvable Leibniz algebras.

Using the result of \cite{Mub}, an approach to the study of
solvable Lie algebras in an arbitrary finite-dimension through the use of the nilradical was developed in
\cite{Ancochea1, NdWi, Rubin, WaLiDe}, etc. In particular, L.Garc{\'{\i}}a  \cite{gar} studied solvable Lie algebras with quasi-filiform nilradicals.  In fact, there are solvable Lie algebras constructed using the method explained in \cite{Mub}.

The analogue of Mubarakzjanov's result has been applied to the Leibniz
algebras in \cite{Nulfilrad}, showing the importance of
consideration of the nilradical in the case of Leibniz algebras as
well. Papers \cite{Abdurasulov1, Abdurasulov, Nulfilrad, Abror2, Shab, Shab22, Shab4, Shab5}  are also devoted
to the study of solvable Leibniz algebras by considering the
nilradical.

The aim of this article is to describe solvable Leibniz algebras with naturally graded quasi-filiform Leibniz nilradicals and with the maximal dimension of complementary space of its nilradical. Namely, naturally graded quasi-filiform Leibniz algebras in any finite dimension over $\mathbb{C}$ were studied by Camacho, G\'{o}mez, Gonz\'{a}lez, and Omirov \cite{Camacho2}. They found five such algebras of the first type, where two of them depend on a parameter and eight algebras of the second type with one of them depending on a parameter. The naturally graded quasi-filiform Lie algebras  were classified in \cite{gomez}. Here exist six families, two of them are decomposable, i.e., split into a direct sum of ideals and as well as there exist some special  cases that appear only in low dimensions.

It is  known that in works devoted to the classification of solvable Leibniz algebras generated by their nilradicals, algebras with  certain nilradical have been studied. In this work, algebras that their nilradicals are isomorphic to  quasi-filiform algebras are studied. It should be noted that the previous result was used directly to obtain the solvable algebra, so the computational processes were much simpler than in the previous work.

Throughout the paper vector spaces and algebras are finite-dimensional over the field of the complex numbers. Moreover, in the table of multiplication of any algebra the omitted products are assumed to be zero
and, if it is not noted, we consider non-nilpotent solvable algebras.

\section{Preliminaries}

In this section we recall some basic notions and concepts used throughout the paper.

\begin{defn}
A vector space with a bilinear bracket $(L,[\cdot,\cdot])$
is called a Leibniz algebra if for any $x,y,z\in L$ the so-called
Leibniz identity
$$\big[x,[y,z]\big]=\big[[x,y],z\big]-\big[[x,z],y\big]$$
holds.
\end{defn}


Further, we use the notation
\[ {\mathcal LI}(x, y, z)=[x,[y,z]] - [[x,y],z] + [[x,z],y].\]
It is obvious that Leibniz algebras are determined by the identity
${\mathcal LI}(x, y, z)=0$.

The sets $Ann_r(L)=\{x \in L: [y,x]=0, \ \forall y \in L\}$ and
$Ann_l(L)=\{x \in L: [x,y]=0, \ \forall y \in L\}$
 are called the right and left annihilators of $L$, respectively.
It is observed that $Ann_r(L)$ is a
two-sided ideal of $L$, and for any $x, y \in L$ the elements $[x,x]$ and $[x,y]+
[y,x]$ belong to $Ann_r(L)$.

For a given Leibniz algebra $(L,[\cdot,\cdot])$, the sequences of
two-sided ideals are defined recursively as follows:
\[L^1=L, \ L^{k+1}=[L^k,L],  \ k \geq 1, \qquad \qquad
L^{[1]}=L, \ L^{[s+1]}=[L^{[s]},L^{[s]}], \ s \geq 1.
\]

These are said to be the lower central and the derived series of $L$,
correspondingly.

\begin{defn} A Leibniz algebra $L$ is said to be
nilpotent (respectively, solvable), if there exists $n\in\mathbb
N$ ($m\in\mathbb N$) such that $L^{n}=\{0\}$ ($L^{[m]}=\{0\}$).
\end{defn}

It is easy to see that the sum of two nilpotent ideals is also nilpotent. Therefore, the maximal nilpotent ideal always exists.
The maximal nilpotent ideal of a Leibniz algebra is said to be the nilradical of the algebra.

\begin{defn} A linear map $d \colon L \rightarrow L$ of a Leibniz algebra $(L,[\cdot,\cdot])$ is said to be a
 derivation if for any $x, y \in L$, the following  condition holds:
 \begin{equation*}\label{der} d([x,y])=[d(x),y] + [x, d(y)] \,.\end{equation*}
\end{defn}
The set of all derivations of $L$ is denoted by $\Der(L),$ which is a Lie algebra with respect to the commutator.

For a given element $x$ of a Leibniz algebra $L$, the right
multiplication operator $\mathcal{R}_x \colon L \to L$ defined by $\mathcal{R}_x(y)=[y,x],  y \in
L$ is a derivation. In fact, Leibniz algebras are characterized by this property regarding right multiplication operators. As in the Lie case, these kinds of derivations are said to be inner derivations.

\begin{defn}
Let $d_1, d_2, \dots, d_n$ be derivations of a Leibniz algebra
$L.$ The derivations $d_1, d_2, \dots, d_n$ are said to be linearly
nil-independent, if for $\alpha_1, \alpha_2, \dots ,\alpha_n \in \mathbb{C}$
and a natural number $k$,
$$(\alpha_1d_1 + \alpha_2d_2+ \dots +\alpha_nd_n)^k=0\ \hbox{implies }\ \alpha_1 = \alpha_2= \dots =\alpha_n=0.$$
\end{defn}

Note that in the above definition the power is understood with respect to composition.

Let $L$ be a solvable Leibniz algebra. Then it can be written
in the form $L=N \oplus Q,$ where $N$ is the nilradical and $Q$ is the
complementary subspace.

Let $R$ be a solvable Leibniz algebra with nilradical $N$. We denote by $Q$ the complementary vector space  of the nilradical $N$ of the algebra $R$.

\begin{thm}[\cite{Nulfilrad}\label{nilr}]
Let $L$ be a solvable Leibniz algebra and $N$ be its nilradical. Then
the dimension of $Q$ is not
greater than the maximal number of nil-independent derivations of
$N.$
\end{thm}

Thanks to Theorem \ref{nilr} we know that for any $x \in Q$, the operator $\mathcal{R}_{{x |}_{N}}$ is a non-nilpotent  derivation of $N$. Let us fix the following notations:
$$\begin{array}{lll}
R^{i}_{j}(\alpha,\beta,\gamma)& \mbox{---}& i\mbox{th}\   j\mbox{-dimensional solvable Leibniz algebra with nilradical $\mathcal{L}(\alpha,\beta,\gamma)$.} \\
H^{i}_{j}(\alpha,\beta,\gamma)& \mbox{---}& i\mbox{th}\ j\mbox{-dimensional solvable Leibniz algebra with nilradical $\mathcal{G}(\alpha,\beta,\gamma).$}\\
F^{i}_{j}(*)& \mbox{---}& i\mbox{th}\ j\mbox{-dimensional solvable Leibniz algebra with 2-filiform nilradical.}
\end{array}$$

Below we define the notion of a quasi-filiform Leibniz algebra.
\begin{defn} A Leibniz algebra $L$ is called quasi-filiform if $L^{n-2}\neq\{0\}$ and $L^{n-1}=\{0\}$,
where $n=\dim L.$
\end{defn}

Given an $n$-dimensional nilpotent Leibniz algebra $L$ such that $L^{s-1}\neq\{0\}$ and $L^{s}=\{0\}$, put $L_i=L^i/L^{i+1},\ 1\leq i\leq s-1$, and $gr (L) =L_1 \oplus L_2 \oplus \cdots \oplus L_{s-1}$. Using $[L^i,L^j]\subseteq L^{i+j}$, it is easy to establish that $[L_i,L_j]\subseteq L_{i+j}$. So, we obtain the graded algebra $gr (L)$. If $gr (L)$ and $L$ are isomorphic, then we say that $L$ is \emph{naturally graded}.

Let $x$ be a nilpotent element of the set $L\setminus L^2.$ For the nilpotent operator of right multiplication $\mathcal{R}_x,$
we define a decreasing sequence $C(x)=(n_1, n_2, \dots, n_k)$, where $n=n_1+n_2+\dots+n_k,$  which consists of the dimensions of Jordan
blocks of the operator $\mathcal{R}_x.$ In the set of such sequences we consider the lexicographic order, that is,
$C(x)=(n_1, n_2, \dots, n_k)\leq C(y)=(m_1, m_2, \dots, m_t)$ iff there exists $i\in \mathbb N$ such that $n_j=m_j$ for any
$j < i$ and $n_i < m_i.$
\begin{defn} The sequence $C(L)=\mathop {\max
}\limits_{x \in L\setminus L^2}C(x)$ is called a characteristic sequence of the algebra $L$.
\end{defn}

Let $L$ be an $n$-dimensional naturally graded quasi-filiform non-Lie Leibniz algebra which has the characteristic sequence $(n-2, 1, 1)$ or $(n-2, 2).$ The first case (case 2-filiform) has been studied in \cite{Camacho3} and the second in \cite{Camacho2}.

\begin{thm} An arbitrary $n$-dimensional naturally graded non-split $2$-filiform Leibniz algebra $(n\geq 6)$ is isomorphic to one of the following non-isomorphic algebras:

\[\mu_1:\ [e_i,e_1]=e_{i+1}, \ 1\leq i\leq n-3,\ [e_1, e_{n-1}] =e_{n};\]
\[\mu_2:\ [e_i,e_1]=e_{i+1}, \ 1\leq i\leq n-3,\ [e_1,e_{n-1}]=e_{2}+e_{n}, \ [e_i,e_{n-1}]=e_{i+1}, \ 2\leq i\leq n-3,\]
where $\{e_1,e_2,\dots,e_{n}\}$ is a basis
of the algebra.
\end{thm}

Thanks to \cite{Abdurasulov}, we already have the classification of solvable Leibniz algebras whose nilradical is a
2-filiform Leibniz algebra.

\begin{thm} Let $R$ be a solvable Leibniz algebra whose nilradical is a $n$-dimensional naturally graded non-split $2$-filiform Leibniz algebra. Then $R$ is isomorphic to one of the following pairwise non-isomorphic algebras:
\[F_{n+1}(\mu_1), \quad F_{n+1}^1(\mu_2), \quad F_{n+1}^2(\mu_2), \quad F_{n+1}^3(\mu_2), \quad F_{n+1}^4(\mu_2).\]
\end{thm}

\begin{defn}
A quasi-filiform Leibniz algebra $L$ is called algebra of
the type I (respectively, type II), if there exists a basic element $e_1\in L\backslash
L^2$ such that the operator $\mathcal{R}_{e_1}$ has the form $\left(\begin{array}{ll}
J_{n-2}&0\\
0&J_{2}
\end{array}\right)$ (respectively, $\left(\begin{array}{ll}
J_{2}&0\\
0&J_{n-2}
\end{array}\right)$).
\end{defn}

In the following theorem we give the classification of naturally graded quasi-filiform Leibniz algebras given in \cite{Camacho2}.

\begin{thm} \label{quasi1} An arbitrary $n$-dimensional naturally graded quasi-filiform Leibniz algebra of type I is isomorphic to one of the following pairwise non-isomorphic algebras of the families:
{\small
\begin{longtable}{l|l}
 \hline
$\mathcal{L}^{1,\beta}_n$ & $ [e_i,e_1]=e_{i+1},\ [e_{n-1},e_1]=e_{n}, \ [e_1,e_{n-1}]=\beta e_{n},\ 1\leq i \leq n-3, \ \beta\in \mathbb{C},$\\
 \hline
$\mathcal{L}^{2,\beta}_n$ & $[e_i,e_1]=e_{i+1},\ [e_{n-1},e_1]=e_{n},\ [e_1,e_{n-1}]=\beta e_{n},\  [e_{n-1},e_{n-1}]=e_n, \ 1\leq i \leq n-3,\ \beta\in \{0,1\},$\\
\hline
$\mathcal{L}^{3,\beta}_n$ & $[e_i,e_1]=e_{i+1},\ [e_{n-1},e_1]=e_{n}+e_2,\ [e_1,e_{n-1}]=\beta e_{n},\ 1\leq i \leq n-3, \ \beta\in \{-1,0,1\},$\\
\hline
$\mathcal{L}^{4,\gamma}_n$ & $[e_i,e_1]=e_{i+1},\  [e_{n-1},e_1]=e_{n}+e_2,\ [e_{n-1},e_{n-1}]=\gamma e_n,\ 1\leq i \leq n-3,\
\gamma\neq 0,$ \\
\hline
$\mathcal{L}^{5,\beta,\gamma}_n$ & $[e_i,e_1]=e_{i+1},\ [e_{n-1},e_1]=e_{n}+e_2,\ [e_1,e_{n-1}]=\beta e_{n},\ [e_{n-1},e_{n-1}]=\gamma e_n,\ 1\leq i \leq n-3,$ \\
&$(\beta,\gamma)=(1,1)\ \text{or}\ (2,4),$\\
\hline
\end{longtable}} where $\{e_1,e_2,\dots,e_{n}\}$ is a basis of the algebra.
\end{thm}

\begin{thm} \label{quasi2} An arbitrary $n$-dimensional naturally graded quasi-filiform Leibniz algebra of second II type is isomorphic to one of the following pairwise non-isomorphic algebras of the families:
{\small
\begin{longtable}{l|l}
 \hline
$\mathcal{L}^{1}_n$& $[e_1,e_1]=e_{2},\ [e_i,e_1]=e_{i+1},\  [e_1,e_i]=-e_{i+1},\  3\leq i \leq n-1,$\\
 \hline
$\mathcal{L}^{2}_n$& $[e_1,e_1]=e_{2},\ [e_i,e_1]=e_{i+1},\ 3\leq i \leq n-1,\ [e_{1},e_3]=e_{2}-e_4,\ [e_1,e_i]=-e_{i+1},\ 4\leq i \leq n-1,$\\
\hline
$\mathcal{L}^{3}_n$&$[e_1,e_1]=e_{2},\ [e_{3},e_3]= e_2,\ [e_i,e_1]=e_{i+1},\ [e_1,e_i]=-e_{i+1},\ 3\leq i \leq n-1,$\\
\hline
$\mathcal{L}^{4}_n$& $[e_1,e_1]=e_{2},\ [e_i,e_1]=e_{i+1},\ 3\leq i \leq n-1,\ [e_{1},e_3]=2e_{2}-e_4,\ [e_1,e_i]=-e_{i+1},\ 4\leq i \leq n-1,$\\
&$[e_{3},e_3]= e_2,$ \\
\hline
$\mathcal{L}^{5}_n,$ & $[e_1,e_1]=e_{2},\ [e_i,e_1]=e_{i+1},\ [e_1,e_i]=-e_{i+1},\ [e_i,e_{n+2-i}]=(-1)^{i}e_{n},\ 3\leq i \leq n-1,$\\
$n$ odd & \\
\hline
$\mathcal{L}^{6,\beta}_n $ & $ [e_1,e_1]=e_{2},\ [e_i,e_1]=e_{i+1},\ 3\leq i \leq n-1,\ [e_{1},e_3]=\beta e_2-e_{4},\ \beta\in\{1,2\},$\\
$n$ odd & $[e_1,e_i]=-e_{i+1},\ 4\leq i \leq n-1,\ [e_i,e_{n+2-i}]=(-1)^{i} e_{n},\ 3\leq i \leq n-1,$\\
\hline
$\mathcal{L}^{7,\gamma}_n$& $[e_1,e_1]=e_{2},\ [e_i,e_1]=e_{i+1},\ [e_1,e_i]=-e_{i+1},\ [e_i,e_{n+2-i}]=(-1)^{i} e_{n},\ 3\leq i \leq n-1,$\\
$n$ odd & $[e_{3},e_3]=\gamma e_2,\ \gamma\neq0,$\\
\hline
$\mathcal{L}^{8,\beta,\gamma}_n$& $[e_1,e_1]=e_{2},\ [e_i,e_1]=e_{i+1},\ 3\leq i \leq n-1,\ [e_{1},e_3]=\beta e_2-e_{4},\ [e_1,e_i]=-e_{i+1},\ 4\leq i \leq n-1,$\\
$n$ odd & $[e_{3},e_3]=\gamma e_2,\ [e_i,e_{n+2-i}]=(-1)^{i}e_{n},\ 3\leq i \leq n-1,\ (\beta,\gamma)=(-2,1),(2,1)\ or\ (4,2),$\\
\hline
\end{longtable}}
where $\{e_1,e_2,\dots,e_{n}\}$ is a basis of the algebra.
\end{thm}

The study of naturally graded quasi-filiform Leibniz algebra of  corresponding type in Theorems \ref{quasi1} and \ref{quasi2} can be simplified as follows (see \cite{Camacho}):

\begin{prop}\label{prop3}
Let $L$ be a naturally graded quasi-filiform Leibniz algebra, then
it is isomorphic to one algebra of the non isomorphic families:
$$\mathcal{L}(\alpha,\beta,\gamma): \left\{
    \begin{array}{lll}
        [e_i,e_1]=e_{i+1},& 1\leq i \leq n-3,&\\[1mm]
        [e_{n-1},e_1]=e_{n}+\alpha e_2,& [e_1,e_{n-1}]=\beta e_{n},& [e_{n-1},e_{n-1}]=\gamma e_n,
    \end{array}\right.$$
$$\mathcal{G}(\alpha,\beta,\gamma): \left\{
    \begin{array}{lll}
[e_1,e_1]=e_{2},& [e_i,e_1]=e_{i+1},& 3\leq i \leq n-1,\\[1mm]
[e_{1},e_3]=-e_{4}+\beta e_2, & [e_1,e_i]=-e_{i+1},& 4\leq i \leq n-1,\\[1mm]
[e_{3},e_3]=\gamma e_2,& [e_i,e_{n+2-i}]=(-1)^{i}\alpha e_{n},& 3\leq i \leq n-1,\\[1mm]
\end{array}\right.$$
where $\{e_1,e_2,\dots,e_{n}\}$ is a basis of the algebra and in the algebra $\mathcal{G}(\alpha,\beta,\gamma)$ if $n$ is odd, then $\alpha\in\{0,1\}$, if $n$ is even, then $\alpha=0$.

\end{prop}

\begin{rem}  The algebras given in Theorem \ref{quasi1} and \ref{quasi2} which stated in Proposition \ref{prop3} are of the form:
$$\begin{array}{lllll}
\mathcal{L}(0,\beta,0):=\mathcal{L}^{1,\beta}_n; & \mathcal{L}(0,\beta,1):=\mathcal{L}^{2,\beta}_n; & \mathcal{L}(1,\beta,0):= \mathcal{L}^{3,\beta}_n;& \mathcal{L}(1,0,\gamma):=\mathcal{L}^{4,\gamma}_n; & \mathcal{L}(1,\beta,\gamma):=\mathcal{L}^{5,\beta,\gamma}_n;\\[1mm]
\mathcal{G}(0,0,0):=\mathcal{L}^{1}_n; & \mathcal{G}(0,1,0):=\mathcal{L}^{2}_n; & \mathcal{G}(0,0,1):=\mathcal{L}^{3}_n; & \mathcal{G}(0,2,1):=\mathcal{L}^{4}_n;& \mathcal{G}(1,0,0):=\mathcal{L}^{5}_n;\\[1mm]
\mathcal{G}(1,\beta,0):=\mathcal{L}^{6,\beta}_n;& \mathcal{G}(1,0,\gamma):=\mathcal{L}^{7,\gamma}_n;& \mathcal{G}(1,\beta,\gamma):=\mathcal{L}^{8,\beta,\gamma}_n.\end{array}$$

\end{rem}

Now we give the classification of solvable Leibniz algebras whose nilradical is  naturally graded quasi-filiform Leibniz algebras. Due to Proposition \ref{prop3} we only need to consider solvable Leibniz algebras with nilradicals $\mathcal{L}(\alpha,\beta,\gamma)$ and $\mathcal{G}(\alpha,\beta,\gamma)$.

In order to start the description we need to know the derivations of naturally graded quasi-filiform Leibniz algebras. From \cite{Abdurasulov0} we recall the derivations of the algebras $\mathcal{L}(\alpha,\beta,\gamma)$ and $\mathcal{G}(\alpha,\beta,\gamma)$.

\begin{prop} \label{prop1} An arbitrary $d\in \Der(\mathcal{L}(\alpha,\beta,\gamma))$ has the following form:
$$\left\{\begin{array}{ll}
d(e_1)=\sum\limits_{t=1}^{n}a_te_t, \\[1mm]
d(e_2)=(2a_1+a_{n-1}\alpha)e_2+\sum\limits_{t=3}^{n-2}a_{t-1}e_t+(a_{n-1}+a_{n-1}\beta)e_n,\\[1mm]
d(e_i)=(ia_1+a_{n-1}\alpha)e_i+\sum\limits_{t=i+1}^{n-2}a_{t-i+1}e_t, \quad 3\leq i\leq n-2,\\[1mm]
d(e_{n-1})=\sum\limits_{t=2}^{n}b_te_t, \\[1mm]
d(e_n)=(b_{n-3}-a_{n-3}\alpha)e_{n-2}+(b_{n-1}+a_1+a_{n-1}\gamma-a_{n-1}\alpha(1+\beta))e_n,\\[1mm]
\end{array}\right.$$
where
$$\begin{array}{l}
b_i=a_i\alpha, \ 2\leq i\leq n-4, \ \beta(b_{n-3}-a_{n-3}\alpha)=\gamma(b_{n-3}-a_{n-3}\alpha)=0, \ b_{n-1}\alpha=a_1\alpha+a_{n-1}\alpha^2,\\[1mm]
\gamma b_{n-1}=\gamma(a_1+a_{n-1}\gamma-a_{n-1}\alpha(1+\beta)), \ \ \gamma a_{n-1}=\beta a_{n-1}(\gamma-\alpha(1+\beta)).
\end{array}
$$

\end{prop}

\begin{prop} \label{prop2} Any derivation of the  algebras $\mathcal{G}(\alpha,\beta,\gamma)$ has the following form:

$$\begin{cases}
d(e_1)=\sum\limits_{t=1}^na_{t}e_t, \\[1mm]
d(e_2)=(2a_{1}+a_{3}\beta)e_2, \\[1mm]
d(e_3)=\sum\limits_{t=2}^{n}b_{t}e_t, \\[1mm]
d(e_4)=\gamma a_3e_2+(a_{1}+b_{3})e_4+\sum\limits_{t=5}^{n-1}b_{t-1}e_t+(b_{n-1}-a_{n-1}\alpha)e_{n},\\[1mm]
d(e_i)=((i-3)a_{1}+b_{3})e_i+\sum\limits_{t=i+1}^{n-1}b_{t-i+3}e_t+(b_{n-i+3}-(-1)^ia_{n-i+3}\alpha)e_{n},\ 5\leq i\leq n-1,\\[1mm]
d(e_n)=((n-3)a_{1}+b_{3}-(-1)^{n}a_3\alpha)e_{n}, \\[1mm]
\end{cases}$$
where
$$
2a_3\gamma+b_3\beta=a_1\beta+a_{3}\beta^2,\ \ (1+(-1)^{n})a_{n-1}\alpha=0,\   \ 2b_{3}\gamma=\gamma(2a_1+a_{3}\beta), \ \ b_3\alpha=a_1\alpha-(-1)^na_{3}\alpha^2.
$$
\end{prop}

The following remark describes the maximal dimensions of the complemented spaces to  $\mathcal{L}(\alpha,\beta,\gamma)$ and $\mathcal{G}(\alpha,\beta,\gamma)$.

\begin{rem} From the Propositions \ref{prop1}-\ref{prop2} and using the Theorem \ref{nilr} it follows that for the possible values of the parameters $\alpha, \beta$ and $\gamma$, we derive the following table:

\begin{table}[h!]
\caption{The dimensions of the complemented spaces to  $\mathcal{L}(\alpha,\beta,\gamma)$ and $\mathcal{G}(\alpha,\beta,\gamma)$.}
\label{Codim}
{\small\begin{longtable}{l|l|l}
    \hline
  Algebra & restrictions  & dimensional of \\
  &&complementary space \\
  \hline
    $\mathcal{L}(0,\beta,0)$   & $b_i=0, \ 2\leq i\leq n-4, \ \beta b_{n-3}=0,$& dim$Q\leq 2$ \\[1mm]
  \hline
    $\mathcal{L}(0,0,1)$   & $a_{n-1}=b_i=0, \ 2\leq i\leq n-3, \ b_{n-1}=a_1,$& dim$Q=1$ \\[1mm]
  \hline
    $\mathcal{L}(0,1,1)$   & $b_i=0, \ 2\leq i\leq n-3, \ b_{n-1}=a_1+a_{n-1},$& dim$Q\leq2$ \\[1mm]
  \hline
    $\mathcal{L}(1,-1,0)$   & $b_i=a_i, \ 2\leq i\leq n-3, \ b_{n-1}=a_1+a_{n-1},$& dim$Q\leq2$ \\[1mm]
  \hline
  $\mathcal{L}(1,0,0)$   & $b_i=a_i, \ 2\leq i\leq n-4, \ b_{n-1}=a_1+a_{n-1},$& dim$Q\leq2$ \\[1mm]
  \hline
   $\mathcal{L}(1,1,0)$   & $a_{n-1}=0, \ b_i=a_i, \ 2\leq i\leq n-3, \ b_{n-1}=a_1,$& dim$Q=1$ \\[1mm]
  \hline
   $\mathcal{L}(1,0,\gamma), \ \gamma\neq0$   & $a_{n-1}=0, \ b_i=a_i, \ 2\leq i\leq n-3, \ b_{n-1}=a_1,$& dim$Q=1$ \\[1mm]
  \hline
   $\mathcal{L}(1,1,1)$  & $a_{n-1}=0, \ b_i=a_i, \ 2\leq i\leq n-3, \ b_{n-1}=a_1,$& dim$Q=1$ \\[1mm]
  \hline
   $\mathcal{L}(1,2,4)$  & $a_{n-1}=0, \ b_i=a_i, \ 2\leq i\leq n-3, \ b_{n-1}=a_1,$& dim$Q=1$ \\[1mm]
  \hline
   $\mathcal{G}(0,0,0)$  & & dim$Q\leq2$ \\[1mm]
  \hline
  $\mathcal{G}(0,1,0)$  & $b_{3}=a_1+a_{3}, $& dim$Q\leq2$ \\[1mm]
  \hline
 $\mathcal{G}(0,0,1)$  & $b_{3}=a_1, \ a_{3}=0, $& dim$Q=1$ \\[1mm]
  \hline
 $\mathcal{G}(0,2,1)$  & $b_{3}=a_1+a_{3}, $& dim$Q\leq2$ \\[1mm]
  \hline
$\mathcal{G}(1,0,0)$  & $b_{3}=a_1+a_{3}, \ a_{n-1}=0, $& dim$Q\leq2$ \\[1mm]
  \hline
$\mathcal{G}(1,1,0)$  & $b_{3}=a_1+a_{3}, \ a_{n-1}=0, $& dim$Q\leq2$ \\[1mm]
  \hline
$\mathcal{G}(1,2,0)$  & $b_{3}=a_1, \ a_{3}=0, \ a_{n-1}=0, $& dim$Q=1$ \\[1mm]
  \hline
$\mathcal{G}(1,0,\gamma), \ \gamma\neq0$  & $b_{3}=a_1, \ a_{3}=0, \ a_{n-1}=0, $& dim$Q=1$ \\[1mm]
  \hline
$\mathcal{G}(1,-2,1)$  & $b_{3}=a_1, \ a_{3}=0, \ a_{n-1}=0, $& dim$Q=1$ \\[1mm]
  \hline
$\mathcal{G}(1,2,1)$  & $b_{3}=a_1+a_{3}, \ a_{n-1}=0, $& dim$Q\leq2$ \\[1mm]
  \hline
$\mathcal{G}(1,4,2)$  & $b_{3}=a_1, a_{3}=0,\ a_{n-1}=0, $& dim$Q=1$ \\[1mm]
  \hline
\end{longtable}}
\end{table}

\end{rem}

\begin{rem}\label{rem1}Thus from the above table \ref{Codim} and the obtained results  it can be seen that the classifications of the solvable Leibniz algebras with the nilradical $\mathcal{L}^{1,\beta}_n,  \mathcal{L}^{2,\beta}_n, $ $ \mathcal{L}^{3,1}_n,  \mathcal{L}^{4,\gamma}_n, $ $ \mathcal{L}^{5,\beta,\gamma}_n,  \mathcal{L}^{1}_n, $ $ \mathcal{L}^{2}_n, \mathcal{L}^{3}_n, \mathcal{L}^{4}_n, $ $ \mathcal{L}^{5}_n, \mathcal{L}^{6,2}_n, $ $ \mathcal{L}^{7,\gamma}_n,  \mathcal{L}^{8,-2,1}_n$ and $\mathcal{L}^{8,4,2}_n$ is obtained in the following papers:
\item \begin{itemize}
          \item  The classification of the solvable Leibniz algebra with the nilradical $\mathcal{L}^{1,\beta}_n$ is stated in paper \cite{Abdurasulov1}.

 \item The classification of the solvable Leibniz algebra with the nilradical $\mathcal{L}^{2,1}_n$ is stated in paper \cite{Abror2}.

 \item The classifications of the solvable Leibniz algebras with the nilradicals $\mathcal{L}^{2,0}_n, \ \mathcal{L}^{3,1}_n, \ \mathcal{L}^{4,\gamma}_n, \ \mathcal{L}^{5,\beta,\gamma}_n, $ $ \mathcal{L}^{6,2}_n, \ \mathcal{L}^{7,\gamma}_n, \ \mathcal{L}^{8,-2,1}_n$ and $\mathcal{L}^{8,4,2}_n$ are stated in paper \cite{Abdurasulov0}. Note that there is no solvable Leibniz algebra with nilradicals $\mathcal{L}^{2,0}_n, \ \mathcal{L}^{3,1}_n, \ \mathcal {L}^{4,\gamma}_n, \ \mathcal{L}^{5,1,1}_n, \ \mathcal {L}^{5,2,4}_n,$ ($\mathcal{L}(0,0,1)$, $\mathcal{L}(1,1,0)$, $\mathcal{L}(1,0,\gamma)$, $\mathcal{L}(1,1,1)$, $\mathcal{L}(1,2,4)$,  respectively).

 \item The classifications of the solvable Leibniz algebra with the nilradicals $\mathcal{L}^{1}_n, \mathcal{L}^{2}_n, \mathcal{L}^{3}_n, \mathcal{L}^{4}_n$ and $\mathcal{L}^{5}_n$ are stated in papers in papers  \cite{Shab}, \cite{Shab22}, \cite{Shab}, \cite{Shab4} and \cite{Shab5}, respectively.

  \end{itemize}
\end{rem}

From the  Remark \ref{rem1}  it can be seen that the classifications of the solvable Leibniz algebras with the nilradical $\mathcal{L}^{3,-1}_n, \ \mathcal{L}^{3,0}_n, \ \mathcal {L}^{6,1}_n, \ \mathcal{L}^{8,2,1}_n$ ($\mathcal{L}(1,-1,0)$, $\mathcal{L}(1,0,0)$, $\mathcal{G}(1,1,0)$, $\mathcal{G}(1,2,1)$, respectively) and the dimension of complememtary space equals one are not described.

\section{Main results. One dimensional extensions of the naturally graded quasi-filiform Leibniz
algebras $\mathcal{L}(1,-1,0)$ and $\mathcal{L}(1,0,0)$.}

\begin{thm}\label{thm1} An arbitrary solvable Leibniz algebra with a codimension one nilradical $\mathcal{L}(1,-1,0)$ is isomorphic to one of the following non-isomorphic algebras:
$$R_{n+1}^1(1,-1,0), \ R_{n+1}^2(1,-1,0).$$
\end{thm}

\begin{proof} Let $R$ be a solvable Leibniz algebra with nilradical $\mathcal{L}(1,-1,0)$ and let $\{e_1, e_2, \dots, e_{n},x\}$ be a basis of the algebra $R$. Then using the form of derivation in the Proposition \ref{prop1} for algebra  $\mathcal{L}(1,-1,0)$ and since $e_1,e_{n-1}\notin Ann_{r}(R), \ e_2,\dots, e_{n-2}\in Ann_{r}(R)$ we have the products in the algebra $R$:

$$\begin{cases}
[e_i,e_1]=e_{i+1}, \ [e_{n-1},e_1]=e_{n}+e_2,\ [e_1, e_{n-1}]=-e_{n}, & 1\leq i \leq n-3,\\[1mm]
[e_1,x]=\sum\limits_{t=1}^{n}a_te_t, \quad [e_i,x]=(ia_1+a_{n-1})e_i+\sum\limits_{t=i+1}^{n-2}a_{t-i+1}e_t, & 2\leq i\leq n-2,\\[1mm]
[e_{n-1},x]=\sum\limits_{t=2}^{n-3}a_te_t+b_{n-2}e_{n-2}+(a_1+a_{n-1})e_{n-1}+b_ne_n, &
[e_n,x]=(2a_1+a_{n-1})e_n,\\[1mm]
[x,e_1]=-a_1e_1+\sum\limits_{t=2}^{n-2}c_te_t-a_{n-1}e_{n-1}+c_ne_n, \quad [x, e_i]=0, & 2\leq i\leq n-2,\\[1mm]
[x,e_{n-1}]=\sum\limits_{t=2}^{n-2}d_te_t-(a_1+a_{n-1})e_{n-1}+d_ne_n, \\[1mm]
[x,e_n]=-(a_{1}+a_{n-1})e_2+\sum\limits_{t=3}^{n-2}d_{t-1}e_t-(2a_1+a_{n-1})e_n,& [x,x]=\sum\limits_{t=2}^{n-2}\alpha_te_t+\alpha_ne_n. \\[1mm]
\end{cases}$$

The equality ${\mathcal L}(x,e_{n},e_1)=0$ implies
$$a_{n-1}=-a_1,\quad d_{t}=0,\ 2\leq t\leq n-4.$$

It is known that $a_1\neq 0$, otherwise the operator $ \mathcal{R}_x $ will be nilpotent. Then by rescaling $x^\prime=\frac{1}{a_1}x$  we can assume $a_{1}=1$. Applying the basis transformations in the following form:

$$e_i'=e_i+\sum\limits_{t=i+1}^{n-2}A_{t-i+1}e_t,\ \ 1\leq i\leq n-2,$$
$$e_{n-1}'=e_{n-1}+\sum\limits_{t=2}^{n-3}A_{t}e_t+B_1e_{n-2}+B_2e_n,\quad e_n'=e_n,\ \ x'=x,$$
with
$$A_2=-a_2,\ \ A_i=-\frac{1}{i-1}(a_i+\sum\limits_{t=2}^{i-1}A_ta_{i-t+1}),\ \ 3\leq i\leq n-3,$$
$$B_{1}=-\frac{1}{n-3}(b_{n-2}+\sum\limits_{t=2}^{n-3}A_ta_{n-t-1}),\ \ \ B_2=-b_n,$$
$$A_{n-2}=-\frac{1}{n-4}(a_{n-2}+B_1+\sum\limits_{t=2}^{i-1}A_ta_{i-t+1}),$$
we obtain
$$a_i=0,\ 2\leq i\leq n-2, \quad b_{n-2}=b_n=0.$$

Taking the change
$$x'=x-\sum\limits_{t=3}^{n-2}c_{t}e_{t-1}-c_ne_{n-1}-\frac{\alpha_{n-2}}{n-3}e_{n-2}-\alpha_ne_n,$$
 we derive $\alpha_{n-2}=\alpha_n=c_n=0,\ c_t=0,\ 3\leq t\leq n-2$.

Now considering the equalities ${\mathcal LI}(x,e_{n-1},x)={\mathcal L}(x,e_1,x)={\mathcal L}(x,x,e_1)=0$, we derive the restrictions:
$$d_{n-3}=d_{n-2}=d_n=\alpha_{t}=a_n=0, \quad 2\leq t\leq n-3.$$

Thus, the table of multiplications of the algebra $R$ has form:
$$\left\{\begin{array}{lll}
[e_1,x]=e_1-e_{n-1},& [e_i,x]=(i-1)e_i, & 2\leq i\leq n-2,\\[1mm]
[e_n,x]=e_n,& [x,e_1]=-e_1+c_2e_2+e_{n-1}, & [x,e_{n}]=-e_n. \\[1mm]
\end{array}\right.$$

Using the multiplication table of the algebra $R$, it is sufficient to consider the following basis change:
$$e_i^\prime=A^{i}e_i,\quad 1\leq i\leq n-2, \quad e_{n-1}^\prime=e_{n-1}, \quad e_{n}^\prime=A(1-A)e_2+Ae_{n}, \quad x^\prime=x.$$

From the product $[x^\prime,e_1^\prime]=-e_1^\prime+c_2^\prime e_2^\prime+e_{n-1}^\prime$, we obtain the following relation:
$$c_2^\prime=\frac{c_2-1+A}{A}.$$

If $c_2\neq1$, then by putting $A=1-c_1$, we get $c_2'=0$ and we have the algebra $R_{n+1}^1(1,-1,0)$;

If $c_2=1$, we obtain the algebra $R_{n+1}^2(1,-1,0)$.
\end{proof}

\begin{thm}\label{thm2} An arbitrary solvable Leibniz algebra with a codimension one nilradical $\mathcal{L}(1,0,0)$ is isomorphic to one of the following pairwise non-isomorphic algebras:
$$R_{n+1}^1(1,0,0),\ R_{n+1}^2(1,0,0),\ R_{n+1}^3(1,0,0), \ R_{n+1}^4(1,0,0),\ R_{n+1}^5(1,0,0),\ R_{n+1}^6(1,0,0).$$

\end{thm}

\begin{proof} Let $R$ be a solvable Leibniz algebra with nilradical $\mathcal{L}(1,0,0)$ and let $\{e_1, e_2, \dots, e_{n},x\}$ be a basis of the algebra $R$. Then using the form of derivation in the Proposition \ref{prop1} for algebra  $\mathcal{L}(1,0,0)$ and since $e_1\notin Ann_{r}(R), \ e_2,\dots, e_{n-2},e_{n}\in Ann_{r}(R)$ we have the products in the algebra $R$:
$$\begin{cases}
[e_i,e_1]=e_{i+1}, \ [e_{n-1},e_1]=e_{n}+e_2, &1\leq i \leq n-3,\\[1mm]
[e_1,x]=\sum\limits_{t=1}^{n}a_te_t, \ [e_2,x]=(2a_1+a_{n-1})e_2+\sum\limits_{t=3}^{n-2}a_{t-1}e_t+a_{n-1}e_n,\\[1mm]
[e_i,x]=(ia_1+a_{n-1})e_i+\sum\limits_{t=i+1}^{n-2}a_{t-i+1}e_t, & 3\leq i\leq n-2,\\[1mm]
[e_{n-1},x]=\sum\limits_{t=2}^{n-4}a_te_t+b_{n-3}e_{n-3}+b_{n-2}e_{n-2}+(a_1+a_{n-1})e_{n-1}+b_ne_n, \\[1mm]
[e_n,x]=(b_{n-3}-a_{n-3})e_{n-2}+2a_1e_n,\ [x,e_1]=-a_1e_1+\sum\limits_{t=2}^{n}c_te_t, \\[1mm]
[x,e_{n-1}]=\sum\limits_{t=2}^{n}d_te_t, \ [x,x]=\sum\limits_{t=2}^{n}\alpha_te_t. \\[1mm]
\end{cases}$$

From the equality ${\mathcal LI}(x, e_{n-1}, e_{1})=0$, we derive the restrictions:
$$d_{n-1}=d_{t}=0,\quad 2\leq t\leq n-3.$$

Note that $(a_1,a_{n-1})\neq (0,0)$, otherwise $a_1 = a_{n-1} = 0$ and then we get a contradiction to the non-nilpotency of the derivation $\mathcal{R}_{x|_N}.$ Now we are going to discuss the possible cases of the parameters $a_1$ and $a_{n-1}$.

\emph{Case 1.} Let $a_1\neq 0$. Then by choosing $x'=\frac{1}{a_1}x$ we can assume $a_1=1$.  Again applying the basis transformation in the following form:
$$e_1'=e_1+\sum\limits_{t=2}^{n-2}A_{t}e_t+A_ne_n,\quad e_i'=e_i+\sum\limits_{t=i+1}^{n-2}A_{t-i+1}e_t,\ \ 2\leq i\leq n-2,$$
$$e_{n-1}'=e_{n-1}+\sum\limits_{t=2}^{n-2}A_{t}e_t+B_1e_n,\quad e_n'=e_n,\ \ x'=x,$$
with
$$A_2=-a_2,\ \ A_i=-\frac{1}{i-1}(a_i+\sum\limits_{t=2}^{i-1}A_ta_{i-t+1}),\ \ 3\leq i\leq n-2,$$
$$B_{1}=-(b_n+A_2a_{n-1}),\ \ A_n=-(A_2a_{n-1}-B_1a_{n-1}+a_n),$$
we obtain $a_i=a_n=b_n=0$ for $2\leq i\leq n-2.$

Changing the basis
$$x'=x-\sum\limits_{t=3}^{n-2}c_{t}e_{t-1}-c_ne_{n-1}-\frac{\alpha_{n}}{2}e_n,$$
yields $\alpha_n=c_n=c_t=0$ for $3\leq t\leq n-2$.

Considering the Leibniz identity, we obtain the following restrictions on structure
constants:
$$\begin{array}{llll}
{\mathcal LI}(x,e_{n-1},x)=0, & \Rightarrow & (a_{n-1}-1)d_n=0,\ (n-3)d_{n-2}+d_nb_{n-3}=0,&  \\[1mm]
{\mathcal LI}(x,e_{1},x)=0,& \Rightarrow &\alpha_{n-4}=c_{n-1}b_{n-3}, \ \alpha_{n-3}=c_{n-1}b_{n-2}-a_{n-1}d_{n-2},\ \alpha_{n-1}=c_2(1+a_{n-1}),&\\[1mm]
&&a_{n-1}(c_{n-1}-1)=0,\ c_2a_{n-1}-\alpha_{n-1}=a_{n-1}d_n,\ \alpha_{t}=0, \ 2\leq t\leq n-5,&\\[1mm]
{\mathcal LI}(x,x,e_{1})=0,& \Rightarrow & d_n=-c_2,\ (a_{n-1}-1)c_2=0,\ (n-3)d_{n-2}-c_2b_{n-3}=0,&\\[1mm]
&& c_2(a_{n-1}+c_{n-1})=0, \ \alpha_{n-3}=c_{n-1}(b_{n-2}+d_{n-2}),\  \alpha_{n-1}=2c_2. &  \\[1mm]
\end{array}$$

If $a_{n-1}\neq 1$, then we get $c_2=0$.

Let $a_{n-1}=1$. Then we have $c_{n-1}=1$ and this implies $c_2=0$. So it is always $c_2=0$. Therefor from the above restrictions we obtain $d_n=d_{n-2}=\alpha_{n-1}=0$.

Thus, the table of multiplications of the algebra $R$ has form:
{\small\begin{equation}\label{eq222}\begin{cases}
[e_i,e_1]=e_{i+1}, \ [e_{n-1},e_1]=e_{n}+e_2,&1\leq i \leq n-3,\\[1mm]
[e_1,x]=e_1+a_{n-1}e_{n-1}, \ [e_2,x]=(2+a_{n-1})e_2+a_{n-1}e_n,\ [e_i,x]=(i+a_{n-1})e_i, & 3\leq i\leq n-2,\\[1mm]
[e_{n-1},x]=b_{n-3}e_{n-3}+b_{n-2}e_{n-2}+(1+a_{n-1})e_{n-1}, \ [e_n,x]=b_{n-3}e_{n-2}+2e_n,\\[1mm] [x,e_1]=-e_1+c_{n-1}e_{n-1}, \ [x,x]=c_{n-1}b_{n-3}e_{n-4}+c_{n-1}b_{n-2}e_{n-3}+\alpha_{n-2}e_{n-2}, \\[1mm]
\end{cases}\end{equation}}
where $a_{n-1}(c_{n-1}-1)=0.$

\emph{Case 1.1.} Let $a_{n-1}\neq 4-n,3-n,2-n,0$. Then we get $c_{n-1}=1$ and taking the change of elements $\{e_1, e_2, e_{n-1}, e_n, x\}$ in (\ref{eq222}) as follows:

$$e_1'=e_1-\frac{a_{n-1}b_{n-3}}{(n-4)(n-4+a_{n-1})}e_{n-3}-\frac{a_{n-1}b_{n-2}}{(n-3)(n-3+a_{n-1})}e_{n-2},\quad e_2'=e_2-\frac{a_{n-1}b_{n-3}}{(n-4)(n-4+a_{n-1})}e_{n-2},$$
$$e_{n-1}'=e_{n-1}-\frac{b_{n-3}}{n-4}e_{n-3}-\frac{b_{n-2}}{n-3}e_{n-2},\quad e_n'=e_n-\frac{b_{n-3}}{n-4+a_{n-1}}e_{n-2},$$
$$x'=x-\frac{b_{n-3}}{n-4+a_{n-1}}e_{n-4}-\frac{b_{n-2}}{n-3+a_{n-1}}e_{n-3}-\frac{\alpha_{n-2}}{n-2+a_{n-1}}e_{n-2},$$
we can assume that $b_{n-3}=b_{n-2}=\alpha_{n-2}=0.$ So, we obtain the family of algebras $R_{n+1}^1(1,0,0)$, where $a_{n-1}\notin\{4-n,3-n,2-n,0\}.$

\emph{Case 1.2.} Let $a_{n-1}=0$.  Then setting
$$e_{n-1}'=e_{n-1}-\frac{b_{n-3}}{n-4}e_{n-3}-\frac{b_{n-2}}{n-3}e_{n-2},\quad e_n'=e_n-\frac{b_{n-3}}{n-4}e_{n-2},$$
$$x'=x-\frac{c_{n-1}b_{n-3}}{n-4}e_{n-4}-\frac{c_{n-1}b_{n-2}}{n-3}e_{n-3}-\frac{\alpha_{n-2}}{n-2}e_{n-2},$$
in (\ref{eq222})  one can get $a_n=b_{n-3}=b_{n-2}=\alpha_{n-2}=\alpha_{n}=0.$

If $c_{n-1}=1$, then we have \emph{Case 1.1} with $a_{n-1}=0$.

If $c_{n-1}\neq1$, then we have the algebra $R_{n+1}^2(1,0,0)$, where $c_{n-1}\in \mathbb{C}\setminus\{1\}$.

\emph{Case 1.3.} Let $a_{n-1}=4-n$. Then we get $c_{n-1}=1$ and applying the transformation
$$e_1'=e_1-\frac{(4-n)b_{n-2}}{n-3}e_{n-2},\quad e_{n-1}'=e_{n-1}-\frac{b_{n-2}}{n-3}e_{n-2},$$
$$x'=x-b_{n-2}e_{n-3}-\frac{\alpha_{n-2}}{2}e_{n-2},$$
we can assume $b_{n-2}=\alpha_{n-2}=0.$ If $b_{n-3}=0,$ then we have case \emph{Case 1.1} with $a_{n-1}=4-n$. If $b_{n-3}\neq0,$ then applying the change of basis $e'_i=A^ie_i, (1\leq i\leq n-2), e'_{n-1}=Ae_{n-1}, e'_{n}=A^2e_{n},$ we derive $b_{n-3}=1$ and obtain the algebra $R_{n+1}^3(1,0,0),$ where $A=\sqrt[n-4]{b_{n-3}}.$

\emph{Case 1.4.} Let $a_{n-1}=3-n$. Then $c_{n-1}=1$ and taking the change of basis elements $\{e_1, e_2, e_{n-1}, e_n, x\}$ in (\ref{eq222}) as follows
$$e_1'=e_1+\frac{(3-n)b_{n-3}}{n-4}e_{n-3},\quad e_2'=e_2+\frac{(3-n)b_{n-3}}{n-4}e_{n-2},$$
$$e_{n-1}'=e_{n-1}-\frac{b_{n-3}}{n-4}e_{n-3},\quad e_n'=e_n+b_{n-3}e_{n-2},$$
$$x'=x+b_{n-3}e_{n-4}-\alpha_{n-2}e_{n-2},$$
we can assume $b_{n-3}=\alpha_{n-2}=0.$ If $b_{n-2}=0,$ then we have case \emph{Case 1.1} with $a_{n-1}=3-n$. If $b_{n-2}\neq0,$ then applying $e'_i=A^ie_i\ (1\leq i\leq n-2),\  e'_{n-1}=Ae_{n-1},\  e'_{n}=A^2e_{n},$ we get $b_{n-2}=1$ and derive Leibniz algebra $R_{n+1}^4(1,0,0),$ where $A=\sqrt[n-3]{b_{n-2}}.$

\emph{Case 1.5.} Let $a_{n-1}=2-n$. Then we get $c_{n-1}=1$ and  applying the transformation
$$e_1'=e_1+\frac{(2-n)b_{n-3}}{2(n-4)}e_{n-3}+\frac{(2-n)b_{n-2}}{n-3}e_{n-2},\quad e_2'=e_2+\frac{(2-n)b_{n-3}}{2(n-4)}e_{n-2},$$
$$e_{n-1}'=e_{n-1}-\frac{b_{n-3}}{n-4}e_{n-3}-\frac{b_{n-2}}{n-3}e_{n-2},\quad e_n'=e_n+\frac{b_{n-3}}{2}e_{n-2},$$
$$x'=x+\frac{b_{n-3}}{2}e_{n-4}+b_{n-2}e_{n-3},$$
we have $b_{n-3}=b_{n-2}=0.$ If $\alpha_{n-2}=0,$ then we have case \emph{Case 1.1} with $a_{n-1}=2-n$. If $\alpha_{n-2}\neq0,$ then putting $e'_i=A^ie_i\ (1\leq i\leq n-2),\  e'_{n-1}=Ae_{n-1},\  e'_{n}=A^2e_{n},$ we get $\alpha_{n-2}=1$ and obtain algebra $R_{n+1}^5(1,0,0),$ where $A=\sqrt[n-2]{\alpha_{n-2}}.$

\emph{Case 2.} Let $a_1=0$. Then $a_{n-1}\neq 0$ and by choosing $x'=\frac{1}{a_{n-1}}x$ we can assume $a_{n-1}=1$.  Taking the change $e_{n-1}'=e_{n-1}+a_ne_n,\ e_n'=e_n,\ x'=x-(b_n-a_n)e_{1},$ we obtain $b_n=a_n=0.$

If again making a change
$$x'=x-\sum\limits_{t=2}^{n-2}c_{t}e_{t-1}-c_ne_{n-1}-\alpha_{n-2}e_{n-2},$$
we derive $\alpha_{n-2}=c_n=c_t=0$ for $2\leq t\leq n-2$.

Considering the Leibniz identity, we obtain the following restrictions on structure
constants:
$$\begin{array}{llll}
{\mathcal LI}(x,e_{n-1},x)=0, & \Rightarrow & d_n=0,&  \\[1mm]
{\mathcal LI}(x,e_{1},x)=0,& \Rightarrow &c_{n-1}=0,\ \alpha_{n-1}=0,\ \alpha_{n-3}=d_{n-2}, \ \alpha_{t}=0, \ \ 2\leq t\leq n-4,&\\[1mm]
{\mathcal LI}(x,x,e_{1})=0,& \Rightarrow & d_{n-2}=0.&\\[1mm]
\end{array}$$

Thus, we obtain the family of algebras $R_{n+1}^6(1,0,0)$.

\end{proof}

Similar to Theorems \ref{thm1} and \ref{thm2} we give the descriptions up to isomorphism of solvable Leibniz algebras with nilradicals $\mathcal{G}(1,1,0), \ \mathcal{G}(1,2,1)$ and one-dimensional complementary space of the nilradical.

From Proposition \ref{prop2} we obtain an arbitrary $d$ derivations of algebras $\mathcal{G}(1,1,0)$ and  $\mathcal{G}(1,2,1)$ have the following form, respectively:

 for the algebra $\mathcal{G}(1,1,0)$
$$\begin{cases}
d(e_1)=\sum\limits_{t=1}^na_{t}e_t, \ \ a_{3}=b_{3}-a_{1},\\[1mm]
d(e_2)=(a_{1}+b_{3})e_2, \\[1mm]
d(e_3)=\sum\limits_{t=2}^{n}b_{t}e_t, \ \ b_{2k+1}=0, \ \ 2\leq k\leq \frac{n-3}{2}, \\[1mm]
d(e_4)=(a_{1}+b_{3})e_4+\sum\limits_{t=5}^{n-1}b_{t-1}e_t+(b_{n-1}-a_{n-1})e_{n},\\[1mm]
d(e_i)=((i-3)a_{1}+b_{3})e_i+\sum\limits_{t=i+1}^{n-1}b_{t-i+3}e_t+(b_{n-i+3}-(-1)^ia_{n-i+3})e_{n},\ 5\leq i\leq n-1,\\[1mm]
d(e_n)=((n-4)a_{1}+2b_{3})e_{n}, \\[1mm]
\end{cases}$$

for the algebra $\mathcal{G}(1,2,1)$
$$\begin{cases}
d(e_1)=\sum\limits_{t=1}^na_{t}e_t, \ \ a_{3}=b_{3}-a_{1},\\[1mm]
d(e_2)=2b_{3}e_2, \\[1mm]
d(e_3)=\sum\limits_{t=2}^{n}b_{t}e_t, \ \ b_{2k+1}=0, \ \ 2\leq k\leq \frac{n-3}{2}, \\[1mm]
d(e_4)=(b_{3}-a_{1})e_2+(a_{1}+b_{3})e_4+\sum\limits_{t=5}^{n-1}b_{t-1}e_t+(b_{n-1}-a_{n-1})e_{n},\\[1mm]
d(e_i)=((i-3)a_{1}+b_{3})e_i+\sum\limits_{t=i+1}^{n-1}b_{t-i+3}e_t+(b_{n-i+3}-(-1)^ia_{n-i+3})e_{n},\ 5\leq i\leq n-1,\\[1mm]
d(e_n)=((n-4)a_{1}+2b_{3})e_{n}, \\[1mm]
\end{cases}$$

\begin{thm}\label{thm3} An arbitrary solvable Leibniz algebra with a codimension one nilradical $\mathcal{G}(1,1,0)$ is isomorphic to one of the following pairwise non-isomorphic algebras:
$$H_{n+1}^1(1,1,0), \ H_{n+1}^2(1,1,0), \ H_{n+1}^3(1,1,0),\ H_{n+1}^4(1,1,0), \ H_{n+1}^5(1,1,0), H_{n+1}^6(1,1,0).$$
\end{thm}

\begin{proof} Let $R$ be a solvable Leibniz algebra with nilradical $\mathcal{G}(1,1,0)$ and let $\{e_1, e_2, \dots, e_{n},x\}$ be a basis of the algebra $R$. Then using the above form of derivation for algebra  $\mathcal{G}(1,1,0)$ and since $e_1,e_3,\dots, e_{n-1}\not \in Ann_{r}(R),\ e_2  \in Ann_{r}(R)$ we have the products in the algebra $R$:
$$\begin{cases}
[e_1,e_1]=e_{2},\ [e_i,e_1]=e_{i+1},& 3\leq i \leq n-1,\\[1mm]
[e_{1},e_3]=-e_{4}+e_2, \ [e_1,e_i]=-e_{i+1},& 4\leq i \leq n-1,\\[1mm]
[e_i,e_{n+2-i}]=(-1)^{i}e_{n},& 3\leq i \leq n-1,\\[1mm]
[e_1,x]=\sum\limits_{t=1}^na_{t}e_t, \ [e_2,x]=(2a_{1}+a_{3})e_2, \ [e_3,x]=b_2e_2+(a_{3}+a_{1})e_3+\sum\limits_{t=4}^{n}b_{t}e_t, \\[1mm]
[e_4,x]=(2a_{1}+a_{3})e_4+\sum\limits_{t=5}^{n-1}b_{t-1}e_t+(b_{n-1}-a_{n-1})e_{n}, \\[1mm]
[e_i,x]=((i-2)a_{1}+a_{3})e_i+\sum\limits_{t=i+1}^{n-1}b_{t-i+3}e_t+(b_{n-i+3}-(-1)^ia_{n-i+3})e_{n},& 5\leq i\leq n-1,\\[1mm]
[e_n,x]=((n-2)a_{1}+2a_{3})e_{n},\ [x,e_1]=-a_1e_1+c_2e_2-\sum\limits_{t=3}^{n-1}a_{t}e_t+c_ne_n, \\[1mm]
[x,e_{3}]=d_2e_2-(a_{3}+a_{1})e_3-\sum\limits_{t=4}^{n-1}b_{t}e_t+d_ne_n, \ b_{2k+1}=0, & 2\leq k\leq \frac{n-3}{2}, \\[1mm]
[x,x]=\alpha_2e_2+\alpha_ne_n. \\[1mm]
\end{cases}$$

From $[x,e_{i+1}]=[x,[e_i,e_1]]=[[x,e_i],e_1]-[[x,e_1],e_i]$ for $3\leq i\leq n-1$, we have
$$\begin{cases}
[x,e_4]=a_1e_2-(2a_{1}+a_{3})e_4-\sum\limits_{t=5}^{n-1}b_{t-1}e_t-(b_{n-1}-a_{n-1})e_{n},\\[1mm]
[x,e_i]=-((i-2)a_{1}+a_{3})e_i-\sum\limits_{t=i+1}^{n-1}b_{t-i+3}e_t-(b_{n-i+3}-(-1)^ia_{n-i+3})e_{n},& 5\leq i\leq n-1,\\[1mm]
[x,e_n]=-((n-2)a_{1}+2a_{3})e_{n}.\end{cases}$$

Consider the following possible cases.

\emph{Case 1.} Let $a_1\neq 0$. Then by choosing $x'=\frac{1}{a_1}x$ we can assume $a_1=1.$ Applying the basis transformation in the following form:

$$e_3'=e_3+A_2e_2+\sum\limits_{t=4}^{n-1}A_{t}e_t,\ \ e_i'=e_i+\sum\limits_{t=i+1}^{n}A_{t-i+3}e_t,\ \ 4\leq i\leq n,$$
with
$$A_2=-B_2,\quad A_4=-b_4,\quad A_i=-\frac{1}{i-3}(b_i+\sum\limits_{t=4}^{i-1}A_tb_{i-t+3}),\ \ 5\leq i\leq n-1,$$
we obtain $b_2=b_i=0$ for $4\leq i\leq n-1.$

Taking the change $x'=x+\sum\limits_{t=4}^{n}a_{t}e_{t-1},$ we derive $a_t=0$ for $4\leq t\leq n$.

Now considering the Leibniz identity, we obtain the following restrictions on structure
constants:
$$\begin{array}{llll}
{\mathcal LI}(x,e_{1},x)=0, & \Rightarrow & c_{2}(1+a_3)=a_2+a_3d_2, \ c_n(n-3+2a_3)=a_3(b_n+d_n),&  \\[1mm]
{\mathcal LI}(x,x,e_{1})=0,& \Rightarrow &c_{2}(1+a_3)=a_2+a_3d_2, \ c_n=-a_3(b_n+d_n),&  \\[1mm]
{\mathcal LI}(x,x,e_{3})=0,& \Rightarrow & d_{2}=0, \ (1+a_3)(b_n+d_n)=0,&  \\[1mm]
{\mathcal LI}(x,e_{3},x)=0,& \Rightarrow & (n-2+2a_3)(b_n+d_n)=0, & \\[1mm]
{\mathcal LI}(x,x,x)=0,& \Rightarrow & \alpha_n(n-2+2a_{3})=0.&  \\[1mm]
\end{array}$$

By using the above restrictions we derive $d_n=-b_n$ and $c_n=0$.  If making change of basis element $e_1'=e_1-c_2e_2$ we can assume $c_2=0$ and the table of multiplications of algebra $R$ has the following form:
$$\begin{cases}
[e_1,e_1]=e_{2},\ [e_i,e_1]=e_{i+1},& 3\leq i \leq n-1,\\[1mm]
[e_{1},e_3]=-e_{4}+e_2, \ [e_1,e_i]=-e_{i+1},& 4\leq i \leq n-1,\\[1mm]
[e_i,e_{n+2-i}]=(-1)^{i}e_{n},& 3\leq i \leq n-1,\\[1mm]
[e_1,x]=e_1+a_3e_3, \ [e_2,x]=(2+a_{3})e_2, \\[1mm]
[e_3,x]=(1+a_{3})e_3+b_{n}e_n, \ [e_i,x]=(i-2+a_{3})e_i,& 4\leq i\leq n-1,\\[1mm]
[e_n,x]=(n-2+2a_{3})e_{n},\ [x,e_1]=-e_1-a_3e_3, \\[1mm]
[x,e_{3}]=-(1+a_{3})e_3-b_ne_n,  \\[1mm]
[x,e_4]=e_2-(2+a_{3})e_4,\ [x,e_i]=-(i-2+a_{3})e_i,& 5\leq i\leq n-1,\\[1mm]
[x,e_n]=-(n-2+2a_{3})e_{n},\ [x,x]=\alpha_2e_2+\alpha_ne_n, \\[1mm]
\end{cases}$$
where $\alpha_n(n-2+2a_{3})=0.$

\emph{Case 1.1.} Let $a_3\neq -2,\ 3-n, \frac{2-n}{2},\frac{3-n}{2}$. Then we have $\alpha_n=0$ and setting
$$e_1'=e_1-\frac{a_3b_n}{(n-3+a_3)(n-3+2a_3)}e_n,\quad e_3'=e_3-\frac{b_n}{n-3+a_3}e_n,\quad x'=x-\frac{\alpha_2}{2+a_3}e_2,$$
we derive $b_2=\alpha_2=0$. Thus, we obtain the algebra $H_{n+1}^1(1,1,0)$.

\emph{Case 1.2} Let $a_3=-2$. Then $\alpha_n=0$  and putting
$$e_1'=e_1+\frac{b_n}{(n-5)^2}e_n,\quad e_3'=e_3-\frac{b_n}{n-5}e_n,$$
we can assume $b_2=0$. If $\alpha_{2}=0,$ then we have algebra in the case \emph{Case 1.1} with $a_{3}=-2$. If $\alpha_{2}\neq0,$ then applying the change of basis $e'_1=Ae_1,\ e'_2=A^2e_2,\ e'_i=A^{i-2}e_i, (3\leq i\leq n),$ we derive $\alpha_{2}=1$ and obtain algebra $H_{n+1}^2(1,1,0),$ where $A=\sqrt{\alpha_2}.$

\emph{Case 1.3.} Let $a_3=3-n$. Then we get $\alpha_n=0$ and putting $x'=x-\frac{\alpha_2}{5-n}e_2$ we have $\alpha_2=0$.
If $b_{n}=0,$ then we have case \emph{Case 1.1} with $a_{3}=3-n$. If $b_{n}\neq0,$ then putting $e'_1=Ae_1,\ e'_2=A^2e_2,\ e'_i=A^{i-2}e_i, (3\leq i\leq n),$ we get $b_{n}=1$ and the algebra $H_{n+1}^3(1,1,0),$ where $A=\sqrt[n-3]{b_{n}}.$

\emph{Case 1.4.}  Let $a_3=\frac{2-n}{2}$. Then applying the transformation
$$e_1'=e_1-\frac{(2-n)b_n}{4-n}e_n,\quad e_3'=e_3+\frac{2b_n}{4-n}e_n, \quad x'=x-\frac{2\alpha_2}{6-n}e_2,$$
we obtain $b_n=\alpha_2=0$.  If $\alpha_{n}=0,$ then we have algebra in the case \emph{Case 1.1} with $a_{3}=\frac{2-n}{2}$. If $\alpha_{n}\neq0,$ then applying $e'_1=Ae_1,\ e'_2=A^2e_2,\ e'_i=A^{i-2}e_i, (3\leq i\leq n),$ we get $\alpha_{n}=1$ and derive Leibniz algebra $H_{n+1}^4(1,1,0),$ where $A=\sqrt[n-2]{\alpha_{n}}.$

\emph{Case 1.5.} Let $a_3=\frac{3-n}{2}$. Then we have $\alpha_n=0$ and putting $x'=x-\frac{2\alpha_2}{7-n}e_2$ we can assume $\alpha_2=0$. If $b_{n}=0,$ then we have case \emph{Case 1.1} with $a_{3}=\frac{3-n}{2}$. If $b_{n}\neq0,$ then putting $e'_1=Ae_1,\ e'_2=A^2e_2,\ e'_i=A^{i-2}e_i, (3\leq i\leq n),$ we get $b_{n}=1$ and obtain algebra $H_{n+1}^5(1,1,0),$ where $A=\sqrt[n-3]{b_{n}}.$

\emph{Case 2.} Let  $a_1=0$. Then we can assume $a_3=1.$  By taking the change of basis element
$$x'=x-a_2e_1-\alpha_2e_2+\sum\limits_{t=4}^{n}a_{t}e_{t-1},$$
we derive $a_2=0,\ \alpha_2=0,\  a_t=0,\ 4\leq t\leq n$. In this case we have that $e_n\not \in Ann_r(R).$

From equalities ${\mathcal LI}(x,e_{1},x)={\mathcal LI}(x,x,e_{3})={\mathcal LI}(x,e_{3},x)=0$  we obtain
$$c_{2}=b_2+d_2, d_{2}=b_2=0.$$

Thus, we obtain the family of algebras $H_{n+1}^6(1,1,0)$, which completes the proof of theorem.
\end{proof}

\begin{thm} An arbitrary solvable Leibniz algebra with a codimension one nilradical $\mathcal{G}(1,2,1)$ is isomorphic to one of the following pairwise non-isomorphic algebras:
$$H_{n+1}^1(1,2,1),\ H_{n+1}^2(1,2,1),\ H_{n+1}^3(1,2,1),\ H_{n+1}^4(1,2,1), \ H_{n+1}^5(1,2,1), \ H_{n+1}^6(1,2,1), \ H_{n+1}^7(1,2,1).$$
\end{thm}
\begin{proof} The proof of this theorem is carried out similarly to the proof of Theorem \ref{thm3}.
\end{proof}

\section{Appendix: The list of the algebras}

{\footnotesize
\begin{longtable}{l|l}
 \hline
$R_{n+1}^1(0,\beta,0)$&$[e_1,x]=e_1+e_{n-1},\ [e_2,x]=2e_2+(1+\beta)e_n,\ [e_i,x]=ie_i,\ 3\leq i\leq n-2,\ [e_{n-1},x]=e_{n-1},\ [e_n,x]=2e_{n},$\\[1mm]
& $[x,e_1]=-e_1-e_{n-1},\ [x,e_{n-1}]=-e_{n-1},\ [x,e_{n}]=(\beta-1)e_n,\ \beta\in\{-1,1\},$\\[1mm]
 \hline
$R_{n+1}^2(0,-1,0)$&$[e_1,x]=e_1+e_n,\ [e_n,x]=e_{n},\ [e_i,x]=ie_i,\ 2\leq i\leq n-2,\ [x,e_1]=-e_1,\ [x,e_{n}]=-e_n,$\\[1mm]
 \hline
$R_{n+1}^3(0,-1,0)$&$[e_i,x]=ie_i,\ 1\leq i\leq n-2,\ [e_{n-1},x]=-e_{n-1},\  [x,e_1]=-e_1,\ [x,e_{n-1}]=e_{n-1},\ [x,x]=e_n,$\\[1mm]
 \hline
$R_{n+1}^4(0,-1,0)$&$[e_i,x]=ie_i,\ 1\leq i\leq n-2,\ [e_{n-1},x]=\alpha e_{n-1},\ [e_n,x]=(1+\alpha)e_{n},$\\[1mm]
&$[x,e_1]=-e_1,\ [x,e_{n-1}]=-\alpha e_{n-1},\ [x,e_{n}]=-(\alpha+1)e_n,$ \\[1mm]
 \hline
$R_{n+1}^5(0,\beta,0)$&$[e_i,x]=ie_i,\ 1\leq i\leq n-2,\  [e_{n-1},x]=\beta e_{n-1},\ [e_n,x]=(1+\beta)e_{n},$\\[1mm]
&$[x,e_1]=-e_1,\ [x,e_{n-1}]=-\beta e_{n-1},\ \beta\notin\{-1,0,1\},$\\[1mm]
 \hline
$R_{n+1}^6(0,-1,0)$&$[e_i,x]=\sum\limits_{t=i+1}^{n-2}\alpha_{t-i+1}e_t,\ 1\leq i\leq n-3,\ [e_{n-1},x]=e_{n-1}+\alpha_{n-1}e_n,\ [e_n,x]=e_{n},$\\[1mm]
&$[x,e_{n-1}]=-e_{n-1}-\alpha_{n-1}e_n,\ [x,e_{n}]=-e_n,\ [x,x]=\alpha_{n}e_{n-2},$\\[1mm]
 \hline
$R_{n+1}^7(0,0,0)$&$[e_i,x]=ie_i,\ 1\leq i\leq n-2,\ [e_{n-1},x]=-e_{n-1},\ [x,e_1]=-e_1,\ [x,x]=e_n,$\\[1mm]
 \hline
$R_{n+1}^8(0,0,0)$&$[e_i,x]=ie_i,\ 1\leq i\leq n-2,\ [e_{n-1},x]=e_{n-3}+(n-3)e_{n-1},\ [e_n,x]=e_{n-2}+(n-2)e_{n},\ [x,e_1]=-e_1,$\\[1mm]
 \hline
$R_{n+1}^9(0,0,0)$&$[e_i,x]=ie_i,\ 1\leq i\leq n-2,\ [e_{n-1},x]=e_{n-2}+(n-2)e_{n-1},\ [e_n,x]=(n-1)e_{n},\ [x,e_1]=-e_1,$\\[1mm]
 \hline
$R_{n+1}^{10}(0,0,0)$&$[e_1,x]=e_1+e_{n},\ [e_i,x]=ie_i,\ 2\leq i\leq n-2,\ [e_n,x]=e_{n},\ \ [x,e_1]=-e_1,\ [x,x]=-e_{n-1},$\\[1mm]
 \hline
$R_{n+1}^{11}(0,0,0)$&$[e_i,x]=ie_i,\ 1\leq i\leq n-2,\ [e_{n-1},x]=\alpha e_{n-1},\ [e_n,x]=(1+\alpha)e_{n},\ [x,e_1]=-e_1,$\\[1mm]
 \hline
$R_{n+1}^{12}(0,0,0)$&$[e_i,x]=\sum\limits_{t=i+1}^{n-2}\alpha_{t-i+1}e_t,\ 1\leq i\leq n-3,\ [e_{n-1},x]=e_{n-1}+\alpha_{n-1}e_n,\ [e_n,x]=e_{n},\ [x,x]=\alpha_{n}e_{n-2},$\\[1mm]
 \hline
$R_{n+2}(0,\beta,0)$& $[e_i,x]=ie_i,\ 1\leq i\leq n-2,\ [e_n,x]=e_n,\ [x,e_1]=-e_1,\ [x,e_n]=\beta e_{n},$\\[1mm]
&$[e_{n-1},y]=e_{n-1},\ [e_n,y]=e_n,\ [y,e_{n-1}]=\beta e_{n-1},\ [y,e_n]=\beta e_{n},\ \beta\in\{-1,0\},$ \\[1mm]
\hline
$R_{n+1}^1(0,1,1)$ &$[e_1,x]=e_1-e_{n-1}+ae_n,\ [e_2,x]=2e_2-2e_n,\ [e_i,x]=ie_i,\ 3\leq i\leq n-2,$ \\[1mm]
&$[e_{n-1},x]=ae_{n},\ [x,e_1]=-e_1+e_{n-1},\ [x,x]=be_n,\ (a,b)\neq(0,0),$ \\[1mm]
\hline
$R_{n+1}^2(0,1,1)$ &$[e_1,x]=e_1+(a-1)e_{n-1},\ [e_2,x]=2e_2+2(a-1)e_n,\ [e_i,x]=ie_i,\ 3\leq i\leq n-2,$ \\[1mm]
&$[e_{n-1},x]=ae_{n-1},\ [e_{n},x]=2ae_{n},\ [x,e_1]=-e_1-(a-1)e_{n-1},\ [x,e_{n-1}]=-ae_{n-1},$\\[1mm]
\hline
$R_{n+1}^3(0,1,1)$ & $[e_1,x]=\sum\limits_{t=2}^{n-2}a_te_t+e_{n-1}-a_2e_n, \ [e_2,x]=\sum\limits_{t=3}^{n-2}a_{t-1}e_t+2e_{n},\ [e_i,x]=\sum\limits_{t=i+1}^{n-2}a_{t-i+1}e_t,\ 3\leq i\leq n-3,$\\[1mm]
&$[e_{n-1},x]=e_{n-1},\ [e_n,x]=2e_n, \ [x,e_{1}]=-e_{n-1},\ [x,e_{n-1}]=-e_{n-1},\ [x,x]=\alpha_{n-2}e_{n-2},$\\[1mm]
\hline
$R_{n+2}(0,1,1)$ &$[e_1,x]=e_1-e_{n-1},\ [e_2,x]=2e_2-2e_n,\ [e_i,x]=ie_i,\ 3\leq i\leq n-2, \ [x,e_1]=-e_1+e_{n-1},$\\[1mm]
 & $[e_1,y]=e_{n-1},\ [e_2,y]=2e_n,\ [e_{n-1},y]=e_{n-1},\ [e_n,y]=2e_n,\ [y,e_1]=-e_{n-1},\ [y,e_{n-1}]=-e_{n-1},$\\[1mm]
\hline
$R_{n+1}^1(1,-1,0)$&$[e_1,x]=e_1-e_{n-1},\ [e_i,x]=(i-1)e_i,\ 2\leq i\leq n-2,\ [e_n,x]=e_n,\ [x,e_1]=-e_1+e_{n-1},\ [x,e_{n}]=-e_n,$\\[1mm]
\hline
$R_{n+1}^2(1,-1,0)$&$[e_1,x]=e_1-e_{n-1},\ [e_i,x]=(i-1)e_i,\ 2\leq i\leq n-2,\ [e_n,x]=e_n,\ [x,e_1]=-e_1+e_2+e_{n-1},\ [x,e_{n}]=-e_n,$ \\[1mm]
\hline
$R_{n+1}^1(1,0,0)$& $[e_1,x]=e_1+ae_{n-1},\ [e_2,x]=(2+a)e_2+ae_n,\ [e_i,x]=(i+a)e_i,\ 3\leq i\leq n-2,$\\[1mm]
&$[e_{n-1},x]=(1+a)e_{n-1},\ [e_n,x]=2e_n,\ [x,e_1]=-e_1+e_{n-1},$\\[1mm]
\hline
$R_{n+1}^2(1,0,0)$& $[e_i,x]=ie_i,\ 1\leq i\leq n-2,\ [e_{n-1},x]=e_{n-1},\ [e_n,x]=2e_n,\ [x,e_1]=-e_1+ce_{n-1},\ c\neq1,$\\[1mm]
\hline
$R_{n+1}^3(1,0,0)$&$[e_1,x]=e_1+(4-n)e_{n-1},\ [e_2,x]=(6-n)e_2+(4-n)e_n,\ [e_i,x]=(i+4-n)e_i,\ 3\leq i\leq n-2,$\\[1mm]
&$[e_{n-1},x]=e_{n-3}+(5-n)e_{n-1},\ [e_n,x]=e_{n-2}+2e_n,\ [x,e_1]=-e_1+e_{n-1},\ [x,x]=e_{n-4},$ \\[1mm]
\hline
$R_{n+1}^4(1,0,0)$&$ [e_1,x]=e_1+(3-n)e_{n-1},\ [e_2,x]=(5-n)e_2+(3-n)e_n,\ [e_i,x]=(i+3-n)e_i,\ 3\leq i\leq n-2,$\\[1mm]
&$[e_{n-1},x]=e_{n-2}+(4-n)e_{n-1},\ [e_n,x]=2e_n,\ [x,e_1]=-e_1+e_{n-1},\ [x,x]=e_{n-3},$ \\[1mm]
\hline
$R_{n+1}^5(1,0,0)$&$[e_1,x]=e_1+(2-n)e_{n-1},\ [e_2,x]=(4-n)e_2+(2-n)e_n,\ [e_i,x]=(i+2-n)e_i,\ 3\leq i\leq n-2,$\\[1mm]
&$[e_{n-1},x]=(3-n)e_{n-1},\ [e_n,x]=2e_n,\ [x,e_1]=-e_1+e_{n-1},\ [x,x]=e_{n-2},$ \\[1mm]
\hline
$R_{n+1}^6(1,0,0)$&$[e_1,x]=\sum\limits_{t=2}^{n-2}a_te_t+e_{n-1},\ [e_2,x]=e_2+\sum\limits_{t=3}^{n-2}a_{t-1}e_t+e_n,\ [e_i,x]=e_i+\sum\limits_{t=i+1}^{n-2}a_{t-i+1}e_t,\  3\leq i\leq n-2,$\\[1mm]
&$[e_{n-1},x]=\sum\limits_{t=2}^{n-4}a_te_t+b_{n-3}e_{n-3}+b_{n-2}e_{n-2}+e_{n-1},\ [e_n,x]=(b_{n-3}-a_{n-3})e_{n-2},\ [x,x]=\alpha_ne_n,$\\[1mm]
\hline
$R_{n+2}(1,0,0)$&$[e_1,x]=e_1-e_{n-1},\ [e_2,x]=e_2-e_n,\ [e_i,x]=(i-1)e_i,\ 3\leq i\leq n-2,\ [e_n,x]=2e_n,\ [x,e_1]=-e_1+e_{n-1},$\\[1mm]
&$[e_1,y]=e_{n-1},\ [e_2,y]=e_2+e_n,\ [e_i,y]=e_i,\ 3\leq i\leq n-2,\ [e_{n-1},y]=e_{n-1},$ \\[1mm]
\hline
$H_{n+1}^1(0,0,0)$ &$[e_1, x] = e_1,\ [e_2, x] = 2e_2,\ [e_i, x] = (i-n)e_i,\ [x, x] = e_n,\ [x, e_1] = -e_1,\ [x, e_i] = (n-i)e_i,\ 3\leq i\leq n-1,$\\[1mm]
\hline
$H_{n+1}^2(0,0,0)$ &$[e_1,x] = ae_2,\ [e_i, x] = e_i + \epsilon e_{i+2} +\sum\limits_{k=i+3}^n
d_{k-i-2}e_k,\ [x, x] = ce_2,$\\[1mm]
&$[x,e_1] = be_2,\ [x, e_i] = -e_i - \epsilon e_{i+2} -\sum\limits_{k=i+3}^n
d_{k-i-2}e_k,\ \epsilon = 0,\pm1,\ 3\leq i\leq n,$\\[1mm]
\hline
$H_{n+1}^3(0,0,0)$ &$[e_1, x]=e_1,\ [e_2, x]=2e_2,\ [e_i, x]=(i-3+a)e_i,\ [x, e_1]=-e_1,\ [x, e_i]=(3-i-a)e_i,\ 3\leq i\leq n,$\\[1mm]
\hline
$H_{n+1}^4(0,0,0)$ &$[e_1, x]=e_1+e_3,\ [e_2, x]=2e_2,\ [e_i, x] = (i-2)e_i,\ [x, e_1]=-e_1-e_3,\ [x, e_i]=(2-i)e_i,\ 3\leq i\leq n,$\\[1mm]
\hline
$H_{n+2}(0,0,0)$ &$[e_1,x]=e_1,\ [e_2,x]=2e_2,\ [e_i,x]=(i-3)e_i,\ [x,e_1]=-e_1,$\\[1mm]
&$[x,e_i]=-(i-3)e_{i},\ [e_i,y]=e_i,\ [y,e_i]=-e_{i},\ 3\leq i\leq n,$\\[1mm]
\hline
$H_{n+1}^1(0,1,0)$&$[e_1,x]=e_1+(a-2)e_3,\ [e_2,x]=ae_2,\ [e_{i},x]=(a+i-4)e_{i},\ 3\leq i\leq n,$\\[1mm]
&$[x,e_1]=-e_1-(a-2)e_3,\ [x,e_3]=(1-a)e_3,\ [x,e_4]=e_2-ae_4,\ [x,e_j]=(4-j-a)e_j,\ 5\leq j\leq n,$\\[1mm]
\hline
$H_{n+1}^2(0,1,0)$&$[e_1,x]=e_1-2e_3+\delta e_5,\ [e_{i},x]=(i-4)e_{i},\ 3\leq i\leq n$,\\[1mm]
&$[x,e_1]=-e_1+2e_3-\delta e_5,\ [x,e_3]=e_3,\ [x,e_4]=e_2,\ [x,e_j]=(4-j)e_j,\ 5\leq j\leq n,\ \delta=\pm1,$\\[1mm]
\hline
$H_{n+1}^3(0,1,0)$&$[e_1,x]=e_1+(2-n)e_3,\ [e_2,x]=(4-n)e_2,\ [e_{i},x]=(i-n)e_{i},\ 3\leq i\leq n-1,\ [x,x]=\delta e_n,\ \delta=\pm1,$\\[1mm]
&$[x,e_1]=-e_1+(n-2)e_3,\ [x,e_3]=(n-3)e_3,\ [x,e_4]=e_2+(n-4)e_4,\ [x,e_j]=(n-j)e_j,\ 5\leq j\leq n-1,$\\[1mm]
\hline
$H_{n+1}^4(0,1,0)$&$[e_1,x]=e_3,\ [e_2,x]=e_2,\ [e_{i},x]=e_{i}+\epsilon e_{i+2}+\sum\limits_{k=i+3}^n{b_{k-i-2}e_k},\ \epsilon=0,\pm1,$\\[1mm]
&$[x,e_1]=-e_3,\ [x,e_i]=-e_i-\epsilon e_{i+2}-\sum\limits_{k=i+3}^n{b_{k-i-2}e_k},\ 3\leq i\leq n,$\\[1mm]
\hline
$H_{n+2}(0,1,0)$ &$[e_1,x]=e_1-e_3,\ [e_2,x]=e_2,\ [e_i,x]=(i-3)e_i,\ 4\leq i\leq n,$\\[1mm]
&$[x,e_1]=-e_1+e_3,\ [x,e_4]=-e_4+e_2,\ [x,e_i]=-(i-3)e_{i},\ 5\leq i\leq n,$\\[1mm]
& $[e_1,y]=e_3,\ [e_i,y]=e_i,\ 2\leq i\leq n,\ [y,e_1]=-e_3,\ [y,e_i]=-e_{i},\ 3\leq i\leq n,$\\[1mm]
\hline
$H_{n+1}(0,0,1)$&$[e_1,x]=e_1,\ [e_2,x]=2e_2,\ [e_i,x]=(i-2)e_i,\ [x,e_1]=-e_1,\ [x,e_i]=-(i-2)e_i,\ 3\leq i\leq n,$\\[1mm]
\hline
$H_{n+1}^1(0,2,1)$&$[e_1,x]=e_1+(a-1)e_3,\ [e_2,x]=2ae_2,\ [e_3,x]=ae_3,\ [e_4,x]=(a-1)e_2+(a+1)e_4,\ [e_{i},x]=(a+i-3)e_{i},$\\[1mm]
&$[x,e_1]=-e_1-(a-1)e_3,\ [x,e_3]=-ae_3,\ [x,e_4]=(a+1)(e_2-e_4),\ [x,e_i]=(3-i-a)e_i,\ 5\leq i\leq n,$\\[1mm]
\hline
$H_{n+1}^2(0,2,1)$&$[e_1,x]=e_1+(2-n)e_3,\ [e_2,x]=2(3-n)e_2,\ [e_3,x]=(3-n)e_3,\ [e_4,x]=(2-n)e_2+(4-n)e_4,\ [e_{i},x]=(i-n)e_{i},$\\[1mm]
&$[x,x]=e_n,\ [x,e_1]=-e_1+(n-2)e_3,\ [x,e_3]=(n-3)e_3,\ [x,e_4]=(4-n)(e_2-e_4),\ [x,e_i]=(n-i)e_i,\ 5\leq i\leq n,$\\
\hline
$H_{n+1}^3(0,2,1)$&$[e_1,x]=e_3,\ [e_2,x]=2e_2,\ [e_4,x]=e_2,\ [e_{i},x]=e_{i}+\epsilon e_{i+2}+\sum\limits_{k=i+3}^n{b_{k-i-2}e_k},$\\[1mm]
&$[x,e_1]=-e_3,\ [x,e_4]=e_2,\ [x,e_i]=-e_i-\epsilon e_{i+2}-\sum\limits_{k=i+3}^n{b_{k-i-2}e_k},\ 3\leq i\leq n,\ \epsilon=0,1,$\\[1mm]
\hline
$H_{n+1}^4(0,2,1)$&$[e_1,x]=e_1-e_3,\ [e_3,x]=de_2,\ [e_4,x]=e_4-e_2,\ [e_{i},x]=(i-3)e_{i},\ [x,x]=\epsilon e_2,$\\[1mm]
&$[x,e_1]=-e_1+(d+f)e_2+e_3,\ [x,e_3]=fe_2,\ [x,e_4]=e_2-e_4,\ [x,e_i]=(3-i)e_i,\ 5\leq i\leq n,$\\[1mm]
&$\epsilon=0,1,$\ if\,\,$\epsilon=0,$\,\,then\,\,$(d, f)\neq(0, 0),$\\[1mm]
\hline
$H_{n+2}^3(0,2,1)$&$[e_1,x]=e_1-e_3,\ [e_4,x]=e_4-e_2,\ [e_i,x]=(i-3)e_i,\ [x,e_1]=-e_1+e_3,\ [x,e_4]=-e_4+e_2,\ [x,e_i]=-(i-3)e_{i},$\\[1mm]
&$[e_1,y]=e_3,\ [e_2,y]=2e_2,\ [e_3,y]=e_3,\ [e_4,y]=e_4+e_2,\ [e_i,y]=e_i,$\\[1mm]
&$[y,e_1]=-e_3,\ [y,e_3]=-e_3,\ [y,e_4]=-e_4+e_2,\ [y,e_i]=-e_{i},\ 5\leq i\leq n,$\\[1mm]
  \hline
$H_{n+1}^1(1,0,0),$&$[e_1,x]=e_1,\ [e_2,x]=2e_2,\  [e_i,x]=(i-3+b)e_i,\  [e_n,x]=(n-4+2b)e_n,$\\[1mm]
$n$ odd&$[x,e_1]=-e_1,\ [x,e_i]=(3-i-b)e_{i},\ [x,e_n]=(4-n-2b)e_{n},\ 3\leq i\leq n-1,$\\[1mm]
  \hline
$H_{n+1}^2(1,0,0),$&$[e_1,x]=e_1+e_n,\ [e_2,x]=2e_2,\ [e_i,x]=\frac{2i-1-n}{2}e_i,\ [e_n,x]=e_n,$ \\[1mm]
$n$ odd&$[x,e_1]=-e_1-e_n,\ [x,e_i]=\frac{n+1-2i}{2}e_{i},\ 3\leq i\leq n-1,$\\[1mm]
  \hline
$H_{n+1}^3(1,0,0),$&$[e_1,x]=e_1,\ [e_2,x]=2e_2,\  [e_i,x]=\frac{2i-2-n}{2}e_i,\ [x,x]=e_n,$\\[1mm]
$n$ odd&$[x,e_1]=-e_1,\ [x,e_i]=\frac{n+2-2i}{2}e_{i},\ 3\leq i\leq n-1,$\\[1mm]
  \hline
$H_{n+1}^4(1,0,0),$ &$[e_1,x]=ae_2,\ [e_i, x]=e_i+\sum\limits_{k=i+1}^{n-1}
c_{k-i}e_{2k-i+1},\ [e_n, x] = 2e_n,\ [x,x] =\epsilon e_2,$\\[1mm]
$n$ odd &$[x,e_1] = be_2,\ [x, e_i] =-e_i-\sum\limits_{k=i+1}^{n-1}
c_{k-i}e_{2k-i+1},\ [x,e_n] = -2e_n,\ \epsilon = 0,1,\ 3\leq i\leq n-1,\ 2k-i+1\leq n,$\\[1mm]
\hline
$H_{n+2}^4(1,0,0),$&$[e_1,x]=e_1-e_3,\ [e_2,x]=2e_2,\ [e_i,x]=(i-3)e_i,\ 4\leq i\leq n,$ \\[1mm]
$n$ odd&$[x,e_1]=-e_1+e_3,\ [x,e_i]=-(i-3)e_{i},\ 4\leq i\leq n,\ [e_1,y]=e_3,\ [e_i,y]=e_i,\ [e_n,y]=2e_n,\ 3\leq i\leq n-1,$ \\[1mm]
&$[y,e_1]=-e_3,\ [y,e_i]=-e_{i},\ [y,e_n]=-2e_n,\ 3\leq i\leq n-1,$\\[1mm]
  \hline
$H_{n+1}^1(1,1,0)$,&$[e_1,x]=e_1+ae_3,\ [e_2,x]=(2+a)e_2,\ [e_i,x]=(i-2+a)e_i,\ 3\leq i\leq n-1,\ [e_n,x]=(n-2+2a)e_{n},$\\[1mm]
$n$ odd&$[x,e_1]=-e_1-ae_3,\ [x,e_{3}]=-(1+a)e_3,\ [x,e_4]=e_2-(2+a)e_4,$\\ [1mm]
&$[x,e_i]=-(i-2+a)e_i,\ [x,e_n]=-(n-2+2a)e_{n},\ 5\leq i\leq n-1,$\\[1mm]
\hline
$H_{n+1}^2(1,1,0)$,&$[e_1,x]=e_1-2e_3,\ [e_i,x]=(i-4)e_i,\ [e_n,x]=(n-6)e_{n},\ 3\leq i\leq n-1,\ [x,e_1]=-e_1+2e_3,\ [x,e_{3}]=e_3,$\\[1mm]
$n$ odd&$[x,e_4]=e_2,\ [x,e_i]=-(i-4)e_i,\ [x,e_n]=-(n-6)e_{n},\ [x,x]=e_2,\ 5\leq i\leq n-1,$\\[1mm]
\hline
$H_{n+1}^3(1,1,0),$&$[e_1,x]=e_1+(3-n)e_3,\ [e_2,x]=(5-n)e_2,\ [e_3,x]=(4-n)e_3+e_n,\ [e_i,x]=(i+1-n)e_i,\ 4\leq i\leq n-1,$\\[1mm]
$n$ odd&$[e_n,x]=(4-n)e_{n},\ [x,e_1]=-e_1-(3-n)e_3,\ [x,e_{3}]=-(4-n)e_3-e_n,$\\[1mm]
&$[x,e_4]=e_2-(5-n)e_4,\ [x,e_i]=-(i+1-n)e_i,\ [x,e_n]=-(4-n)e_{n},\ 5\leq i\leq n-1,$\\[1mm]
\hline
$H_{n+1}^4(1,1,0),$&$[e_1,x]=e_1+(1-\frac{n}{2})e_3,\ [e_2,x]=(3-\frac{n}{2})e_2,\ [e_i,x]=(i-1-\frac{n}{2})e_i,\ 3\leq i\leq n-1,\ [x,x]=e_n,$\\[1mm]
$n$ odd&$[x,e_1]=-e_1-(1-\frac{n}{2})e_3,\ [x,e_{3}]=-(2-\frac{n}{2})e_3,\ [x,e_4]=e_2-(3-\frac{n}{2})e_4,\ [x,e_i]=-(i-1-\frac{n}{2})e_i,\ 5\leq i\leq n-1,$\\[1mm]
  \hline
$H_{n+1}^5(1,1,0)$,&$[e_1,x]=e_1+\frac{3-n}{2}e_3,\ [e_2,x]=(2+\frac{3-n}{2})e_2,\ [e_3,x]=(1+\frac{3-n}{2})e_3+e_n,\ [e_i,x]=(i-2+\frac{3-n}{2})e_i,\ 4\leq i\leq n-1,$\\[1mm]
$n$ odd&$[e_n,x]=e_{n},\ [x,e_1]=-e_1-\frac{3-n}{2}e_3,\ [x,e_{3}]=-(1+\frac{3-n}{2})e_3-e_n,\ [x,e_4]=e_2-(2+\frac{3-n}{2})e_4,$\\[1mm]
&$[x,e_i]=-(i-2+\frac{3-n}{2})e_i,\ 5\leq i\leq n-1,\ [x,e_n]=-e_{n},$\\[1mm]
\hline
$H_{n+1}^6(1,1,0)$,&$[e_1,x]=e_3,\ [e_2,x]=e_2,\ [e_i,x]=e_i+\sum\limits_{t=i+1}^{n}b_{t-i+3}e_t,\ 3\leq i\leq n-1,\ [e_n,x]=2e_{n},$\\[1mm]
$n$ odd&$[x,e_1]=-e_3,\ [x,e_i]=-e_i-\sum\limits_{t=i+1}^{n}b_{t-i+3}e_t,\  3\leq i\leq n-1,\ [x,e_n]=-2e_{n},\ b_{2k+1}=0,\ 2\leq k\leq \frac{n-3}{2},$\\[1mm]
\hline
$H_{n+2}(1,1,0)$,&$ [e_1,x]=e_1-e_3,\ [e_2,x]=e_2,\ [e_i,x]=(i-3)e_i,\ 4\leq i\leq n-1,\ [e_n,x]=(n-4)e_n,\  [x,e_1]=-e_1+e_3,$\\[1mm]
$n$ odd&$ [x,e_4]=-e_4+e_2,\ [x,e_i]=-(i-3)e_{i},\ 5\leq i\leq n-1,\ [x,e_n]=-(n-4)e_{n},\ [e_1,y]=e_3, $\\[1mm]
&$[e_i,y]=e_i,\ [e_n,y]=2e_n,\ 2\leq i\leq n-1,\ [y,e_1]=-e_3,\ [y,e_i]=-e_{i},\ [y,e_n]=-2e_n,\ 3\leq i\leq n-1,$\\[1mm]
  \hline
$H_{n+1}(1,2,0),$& $[e_1,x]=e_1,\ [e_2,x]=2e_2,\ [e_i,x]=(i-2)e_i,\ 3\leq i\leq n,$\\[1mm]
$n$ odd&$[x,e_1]=-e_1,\ [x,e_3]=-e_3,\ [x,e_4]=-2e_4+2e_2,\ [x,e_i]=-(i-2)e_i,\ 5\leq i\leq n,$\\[1mm]
  \hline
$H_{n+1}(1,0,\gamma)$&$[e_1,x]=e_1,\ [e_2,x]=2e_2,\ [e_i,x]=(i-2)e_i,\ [x,e_1]=-e_1,\ [x,e_i]=-(i-2)e_i,\ 3\leq i\leq n,\ \gamma\neq0,$\\[1mm]
$n$ odd, &\\
\hline
$H_{n+1}(1,-2,1),$&$[e_1,x]=e_1,\ [e_2,x]=2e_2,\ [e_i,x]=(i-2)e_i,\ 3\leq i\leq n,$\\[1mm]
$n$ odd&$[x,e_1]=-e_1,\ [x,e_3]=-e_3,\ [x,e_4]=-2e_4-2e_2,\ [x,e_i]=-(i-2)e_i,\ 5\leq i\leq n,$\\[1mm]
\hline
$H_{n+1}^1(1,2,1)$,&$[e_1,x]=e_1+ae_3,\ [e_2,x]=2(1+a)e_2,\ [e_3,x]=(a+1)e_3,\ [e_4,x]=ae_2+(2+a)e_4,$\\[1mm]
$n$ odd&$[e_i,x]=(i-2+a)e_i,\ 5\leq i\leq n-1,\ [e_n,x]=(n-2+2a)e_{n},\ [x,e_1]=-e_1-ae_3,\ [x,e_{3}]=-(a+1)e_3,$\\[1mm]
&$[x,e_4]=(2+a)e_2-(2+a)e_4,\ [x,e_i]=-(i-2+a)e_i,\ [x,e_n]=-(n-2+2a)e_{n},\ 5\leq i\leq n-1,$\\[1mm]
  \hline
$H_{n+1}^2(1,2,1),$&$[e_1,x]=e_1-e_3,\ [e_3,x]=be_2,\ [e_4,x]=-e_2+e_4,\ [e_i,x]=(i-3)e_i,\ 5\leq i\leq n-1,\ [e_n,x]=(n-4)e_{n},$\\[1mm]
$n$ odd&$ [x,e_1]=-e_1+(b+d)e_2+e_3,\ [x,e_{3}]=de_2,\ [x,e_4]=e_2-e_4,\ [x,e_i]=-(i-3)e_i,\ 5\leq i\leq n-1,$\\
& $[x,e_n]=-(n-4)e_{n},\ [x,x]=\alpha e_2,\ (b,d,\alpha)\neq(0,0,0),$ \\[1mm]
  \hline
$H_{n+1}^3(1,2,1)$,&$[e_1,x]=e_1-\frac{1}{2}e_3,\ [e_2,x]=e_2,\ [e_3,x]=\frac{1}{2}e_3,\ [e_4,x]=-\frac{1}{2}e_2+\frac{3}{2}e_4,\ [e_i,x]=\frac{2i-5}{2}e_i,\ 5\leq i\leq n-1,$\\[1mm]
$n$ odd&$[e_n,x]=(n-3)e_{n},\ [x,e_1]=-e_1+e_2+\frac{1}{2}e_3,\ [x,e_{3}]=-\frac{1}{2}e_3,\ [x,e_4]=\frac{3}{2}e_2-\frac{3}{2}e_4,$\\[1mm]
&$[x,e_i]=-\frac{2i-5}{2}e_i,\ 5\leq i\leq n-1,\ [x,e_n]=-(n-3)e_{n},$\\[1mm]
  \hline
$H_{n+1}^4(1,2,1)$,&$[e_1,x]=e_1+(3-n)e_3,\ [e_2,x]=2(4-n)e_2,\ [e_3,x]=(4-n)e_3+e_n,\ [e_4,x]=(3-n)e_2+(5-n)e_4,$\\[1mm]
$n$ odd&$[e_i,x]=(i+1-n)e_i,\ 5\leq i\leq n-1,\ [e_n,x]=(4-n)e_{n},\ [x,e_1]=-e_1-(3-n)e_3,\ [x,e_{3}]=-(4-n)e_3-e_n,$\\[1mm]
& $[x,e_4]=(5-n)e_2-(5-n)e_4,\ [x,e_i]=-(i+1-n)e_i,\ [x,e_n]=-(4-n)e_{n},\ 5\leq i\leq n-1,$\\[1mm]
  \hline
$H_{n+1}^5(1,2,1),$&$ [e_1,x]=e_1+\frac{3-n}{2}e_3,\ [e_2,x]=(5-n)e_2,\ [e_3,x]=\frac{5-n}{2}e_3+e_n,\ [e_4,x]=\frac{3-n}{2}e_2+\frac{7-n}{2}e_4,$\\[1mm]
$n$ odd&$[e_i,x]=\frac{2i-1-n}{2}e_i,\ 5\leq i\leq n-1,\ [e_n,x]=e_{n},\ [x,e_1]=-e_1-\frac{3-n}{2}e_3,$\\[1mm]
&$  [x,e_{3}]=-\frac{5-n}{2}e_3-e_n,\ [x,e_4]=\frac{7-n}{2}e_2-\frac{7-n}{2}e_4,\ [x,e_i]=-\frac{2i-1-n}{2}e_i,\ [x,e_n]=-e_{n},\ 5\leq i\leq n-1,$\\[1mm]
  \hline
$H_{n+1}^6(1,2,1),$&$[e_1,x]=e_1+\frac{2-n}{2}e_3,\ [e_2,x]=(4-n)e_2,\ [e_3,x]=\frac{4-n}{2}e_3,\ [e_4,x]=\frac{2-n}{2}e_2+\frac{6-n}{2}e_4,$\\[1mm]
$n$ odd&$[e_i,x]=\frac{2i-2-n}{2}e_i,\ 5\leq i\leq n-1,\ [x,e_1]=-e_1-\frac{2-n}{2}e_3,\ [x,e_{3}]=-\frac{4-n}{2}e_3,$\\[1mm]
&$[x,e_4]=\frac{6-n}{2}e_2-\frac{6-n}{2}e_4,\ [x,e_i]=-\frac{2i-2-n}{2}e_i,\ 5\leq i\leq n-1,\ [x,x]=e_n,$\\[1mm]
  \hline
$H_{n+1}^7(1,2,1)$,&$[e_1,x]=e_3,\ [e_2,x]=2e_2,\ [e_3,x]=e_3+\sum\limits_{t=5}^{n-1}b_{t}e_t,\ [e_4,x]=e_2+e_4+\sum\limits_{t=6}^{n}b_{t-1}e_t,$\\[1mm]
$n$ odd&$[e_i,x]=e_i+\sum\limits_{t=i+2}^{n}b_{t-i+3}e_t,\ 5\leq i\leq n-1,\ [e_n,x]=2e_{n},\ [x,e_1]=-e_3,$\\[1mm]
&$[x,e_{3}]=-e_3-\sum\limits_{t=5}^{n-1}b_te_t, \ [x,e_4]=e_2-e_4-\sum\limits_{t=6}^{n}b_{t-1}e_t,\
[x,e_i]=-e_i-\sum\limits_{t=i+2}^{n}b_{t-i+3}e_t,\  5\leq i\leq n-1,$\\[1mm]
&$[x,e_n]=-2e_{n},\ b_{2k+1}=0,\ 2\leq k\leq \frac{n-3}{2},$ \\[1mm]
  \hline
$H_{n+2}(1,2,1)$&$[e_1,x]=e_1-e_3,\ [e_4,x]=e_4-e_2,\ [e_i,x]=(i-3)e_i,\ [e_n,x]=(n-4)e_n,$\\[1mm]
$n$ odd,&$[x,e_1]=-e_1+e_3,\ [x,e_4]=-e_4+e_2,\ [x,e_i]=-(i-3)e_{i},\ [x,e_n]=-(n-4)e_{n},$\\[1mm]
&$[e_1,y]=e_3,\ [e_2,y]=2e_2,\ [e_3,y]=e_3,\ [e_4,y]=e_4+e_2,\ [e_i,y]=e_i,\ [e_n,y]=2e_n,$\\[1mm]
&$[y,e_1]=-e_3,\ [y,e_3]=-e_3,\ [y,e_4]=-e_4+e_2,\ [y,e_i]=-e_{i},\ [y,e_n]=-2e_n,\ 5\leq i\leq n-1,$\\[1mm]
  \hline
$H_{n+1}(1,4,2),$&$[e_1,x]=e_1,\ [e_2,x]=2e_2,\ [e_i,x]=(i-2)e_i,\ 3\leq i\leq n,$\\[1mm]
$n$ odd &$[x,e_1]=-e_1,\ [x,e_3]=-e_3,\ [x,e_4]=-2e_4+4e_2,\ [x,e_i]=-(i-2)e_i,\ 5\leq i\leq n,$\\[1mm]
\hline
$F_{n+1}(\mu_1)$&$[e_i,x]=\sum\limits_{j=i+1}^{n-2}a_{j-i+1}e_j, \ 1\leq i\leq n-2,\ [e_{n-1},x]=e_{n-1},\ [e_{n},x]=e_{n},\ [x,e_{n-1}]=-e_{n-1},\ [x,x]=\delta_{n-2}e_{n-2},$\\[1mm]
\hline
$F_{n+1}^1(\mu_2)$&$[e_1,x]=e_{n-1},\ [e_2,x]=e_2+e_{n},\ [e_i,x]=(i-1)e_i, \ 3\leq i\leq n-2,$\\[1mm]
&$[e_{n-1},x]=e_{n-1},\ [x,e_1]=-e_{n-1},\ [x,e_{n-1}]=-e_{n-1},$\\[1mm]
\hline
$F_{n+1}^2(\mu_2)$&$[e_1,x]=e_{n-1},\ [e_2,x]=e_2+e_{n},\ [e_i,x]=(i-1)e_i, \ 3\leq i\leq n-2,$\\[1mm]
&$[e_{n-1},x]=e_{n-1},\ [x,e_1]=-e_{n-1},\ [x,e_{n-1}]=-e_{n-1},\ [x,x]=e_{n},$\\[1mm]
\hline
$F_{n+1}^3(\mu_2)$&$[e_1,x]=e_{n-1},\ [e_2,x]=e_2+e_{n},\ [e_i,x]=(i-1)e_i, \ 3\leq i\leq n-2,$\\[1mm]
&$[e_{n-1},x]=e_{n-1},\ [x,e_1]=-e_{n-1}+e_{n},\ [x,e_{n-1}]=-e_{n-1},\ [x,x]=\gamma e_{n},$\\[1mm]
\hline
$F_{n+1}^4(\mu_2)$&$[e_1,x]=e_{n-1}+e_{n},\ [e_2,x]=e_2+e_{n},\ [e_i,x]=(i-1)e_i, \ 3\leq i\leq n-2,$\\[1mm]
&$[e_{n-1},x]=e_{n-1},\ [x,e_1]=-e_{n-1}+\beta e_{n},\ [x,e_{n-1}]=-e_{n-1},\ [x,x]=\gamma e_{n}.$\\[1mm]
\hline
\end{longtable}
}


\begin{thebibliography}{99}

\bibitem{Abdurasulov0}   Abdurasulov, K.K., Adashev J.K. (2021). {\it Maximal solvable Leibniz algebras whose nilradical is a quasi-filiform algebra.} Submitted to journal. arXiv:2105.13141.

\bibitem{Abdurasulov1}   Abdurasulov, K.K., Adashev J.K., Casas, J.M., Omirov, B.A. (2019). {\it Solvable Leibniz algebras whose nilradical is a quasi-filiform Leibniz algebra of maximum length.} Comm. Algebra 47(4):1578--1594. DOI: 10.1080/00927870903236160.

\bibitem{Abdurasulov} Abdurasulov K.K., Adashev J.K., Sattarov A.M. (2016). {\it Solvable Leibniz algebras with 2-filiform nilradical.} Uzbek Math. J. 4:16--23.

\bibitem{Ancochea1} Ancochea, J.~M., Campoamor-Stursberg, R., Garc{\'{\i}}a, L.  (2011). {\it Classification of Lie algebras with naturally graded quasi-filiform nilradicals.} Journal of Geometry and Physics 61:2168--2186. DOI:10.1016/j.geomphys.2011.06.015.

\bibitem{Bar} Barnes, D.W. (2012). {\it On Levi's theorem for Leibniz algebras.} Bull. Aust. Math. Soc. 86(2):184--185. DOI:10.1017/S0004972711002954.


\bibitem{Bloh} Bloh, A. (1965). {\it On a generalization of the concept of Lie algebra.} Dokl. Akad. Nauk SSSR. 165(3):471--473.

\bibitem{Camacho} Camacho, L.M., Ca\~{n}ete, E.M., G\'{o}mez, J.R., Omirov B.A. (2011). {\it Quasi-filiform Leibniz algebras of maximum length.} Sib. Math. J. 52(5):840--853. DOI. 10.1134/S0037446611050090.

\bibitem{Camacho2} Camacho, L. M.,  G\'{o}mez, J. R., Gonz\'{a}lez, A. J., Omirov, B. A. (2009).{\it Naturally graded quasi-filiform
Leibniz algebras.} J. Symbolic Comput. 44(5):527--539. DOI: 10.1016/j.jsc.2008.01.006.

\bibitem{Camacho3} Camacho, L.M., G\'{o}mez, J.R., Gonz\'{a}lez, A.J., Omirov, B.A. (2010). {\it Naturally graded 2-filiform Leibniz algebras.} Comm. Algebra 38(10):3671--3685. DOI: 10.1080/00927870903236160.

\bibitem{cartan} Cartan, E. (1894).  {\it Sur la structure des groups de transformations finis et continus} (Paris: thesis, Nony), in: Oeuvres Completes, Partie I. Tome 1:137--287.


\bibitem{Nulfilrad} Casas, J.M., Ladra, M., Omirov, B.A., Karimjanov, I.A. (2013). {\it Classification of solvable Leibniz algebras with null-filiform nilradical.} Linear Mult. Alg. 61(6):758--774. DOI: 10.1080/03081087.2012.703194.

\bibitem{gar} Garc{\'{\i}}a, L. (2010). {\it \'{E}tude g\'{e}om\'{e}trique et structures diff\'{e}rentielles g\'{e}n\'{e}ralisees sur les alg\`{e}bres de Lie quasifiliformes complexes et r\'{e}elles.} Tesis doctoral, Universidad complutense de Madrid. 149 p.

\bibitem{gomez} G\'{o}mez, J.R., Jim\'{e}nez-Merch\'{a}n, A. (2002). {\it Naturally graded quasi-filiform Lie algebras.} J. Algebra 256(1):221--228. DOI:10.1016/S0021-8693(02)00130-8.

\bibitem{Abror2} Khudoyberdiyev, A.Kh., Ladra, M., Omirov, B.A. (2014). {\it On solvable Leibniz algebras whose nilradical is a direct
sum of null-filiform algebras.} Linear Mult. Alg. 62(9):1220--1239. DOI: 10.1080/03081087.2013.816305.

\bibitem{Loday} Loday, J.-L. (1993).  {\it Une version non commutative des alg\`ebres de Lie: les alg\`ebres de Leibniz.} Enseign. Math. (2), 39(3-4):269--293.

\bibitem{Mub} Mubarakzjanov, G. M. (1963). {\it On solvable Lie algebras.} (Russian), Izvestiya Vysshikh Uchebnykh Zavedenii.
Matematika. 1:114--123.

\bibitem{NdWi} Ndogmo, J. C., Winternitz, P. (1994). {\it Solvable Lie algebras with abelian nilradicals.} J. Phys. A: Math. Gen.
27(2):405--423. DOI: 10.1088/0305-4470/27/2/024.

\bibitem{Rubin} Rubin, J. L., Winternitz, P. (1993). {\it Solvable Lie algebras with Heisenberg ideals.} J. Phys. A 26(5):1123--1138. DOI: 10.1088/0305-4470/26/5/031.

\bibitem{Shab} Shabanskaya, A. (2017). {\it Solvable extensions of naturally graded quasi-filiform Leibniz algebras of second type $\mathcal{L}^1$ and $\mathcal{L}^3$}. Comm. Algebra 45(10):4492--4520. DOI: 10.1080/00927872.2016.1270294.

\bibitem{Shab22} Shabanskaya, A. (2018). {\it Solvable extensions of the naturally graded quasi-filiform Leibniz algebra of second type $\mathcal{L}^2$.} Comm. Algebra 46(11):5006--5031. DOI: 10.1080/00927872.2018.1459653.

\bibitem{Shab4} Shabanskaya, A. (2020). {\it Solvable extensions of the naturally graded quasi-filiform Leibniz algebra of second type $\mathcal{L}^4$.} Comm. Algebra 48(2):490--507. DOI: 10.1080/00927872.2019.1648652.

\bibitem{Shab5} Shabanskaya, A. (2020). {\it Solvable extensions of the naturally graded quasi-filiform Leibniz algebra of second type $\mathcal{L}^5$.} Comm. Algebra. DOI: 10.1080/00927872.2020.1834574.

\bibitem{WaLiDe} Wang, Y., Lin, J., Deng Sh. (2008). {\it Solvable Lie algebras with quasi-filiform nilradicals.} Comm. Algebra 36(11):4052--4067. DOI: 10.1080/00927870802174629.

\end{thebibliography}
\end{document}